\documentclass[leqno]{amsart}
\usepackage[utf8]{inputenc}
\usepackage[english]{babel}
\usepackage{newclude}
\usepackage{amsfonts,amssymb,amsmath,amsgen, amsthm}
\usepackage{color, xcolor}
\usepackage{enumerate}
\usepackage{centernot}
\usepackage{comment}
\usepackage{hyperref}
\usepackage{soul}

\newtheorem{theorem}{Theorem}[section]
\newtheorem{lem}[theorem]{Lemma}
\newtheorem{prop}[theorem]{Proposition}

\newtheorem{remark}[theorem]{Remark}
\theoremstyle{definition}
\newtheorem{defin}[theorem]{Definition}
\newtheorem{conjecture}[theorem]{Conjecture}

\def\R{\mathbb R}
\def\C{\mathbb C}
\def\N{\mathbb N}
\def\d{\partial}

\newcommand{\ii}{\mathrm{i}\mkern1mu}
\def\eps{\varepsilon}

\title[Dissipative Rotational GPE]{Existence and Large Time Behavior for a Dissipative Variant of the Rotational NLS Equation}
\author{Paolo Antonelli and Boris Shakarov}
\date{\today}

\begin{document}
\maketitle

\begin{abstract}
    We study a dissipative variant of the Gross-Pitaevskii equation with  rotation. The model contains a nonlocal, nonlinear term that forces the conservation of $L^2$-norm of solutions. We are motivated by several physical experiments and numerical simulations studying the formation of vortices in Bose-Einstein condensates. 
	We show local and global well-posedness of this model and investigate the asymptotic behavior of its solutions. In the linear case, the solution asymptotically tends to the eigenspace associated with the smallest eigenvalue in the decomposition of the initial datum. In the nonlinear case, we obtain weak convergence  to a stationary state. Moreover, for initial energies in a specific range, we prove strong asymptotic stability of ground state solutions.
\end{abstract}

\section{Introduction}

\subsection{Physical background and motivations}

Since their first experimental realization \cite{cornell2002nobel, ketterle2002nobel}, Bose-Einstein condensates (BECs) have become the object of intensive research in the scientific community. In particular, the observation and measurement of its superfluid properties, such as quantized vortices \cite{matthews1999vortices, MaChWoDa00, AbRaVoKe01} for instance, stimulated a considerable amount of interest, leading to investigations of their dynamics \cite{barenghi2001quantized}.
\newline
From the theoretical point of view, one of the main features of BECs is that they are accurately described by the Gross-Pitaevskii (GP) equation \cite{gross1963hydrodynamics, Pi61}, see \cite{pitaevskii2016bose} for a comprehensive physical introduction on the model, its derivation and main properties. The GP description allows the prediction and simulation of relevant phenomena observed in experiments \cite{dalfovo1999theory}. In a rotating frame, the (scaled) GP equation may be written as
\begin{equation}\label{eq:gp1}
	\ii \partial_t \psi = -\frac{1}{2}\Delta \psi + V \psi + g|\psi|^{2} \psi - \Omega \cdot L\psi
\end{equation}
with $\psi:\R^{1+d}\to\C$ being the order parameter describing the condensate.
In equation \eqref{eq:gp1}, $g\in\R$ denotes the (scaled) Gross-Pitaevskii constant,
$\Omega\in\R^d$ determines the rotation axis, $L$ is the angular momentum operator defined by
\begin{equation}\label{eq:ang_mom}
	L := -\ii x \wedge \nabla,
\end{equation}
and $V$ is a (possibly anisotropic) harmonic confining potential \cite{Fe08}, given by
\begin{equation}\label{eq:V}
	V(x) =  \frac{1}{2}\sum_{j=1}^d \omega_j^2 x_j^2,
 \quad\omega_j>0.
\end{equation}
Throughout this paper, we consider $d=2, 3$, for  physical motivations. The case $d=2$ describes the effective dynamics of a BEC under the influence of a strongly anisotropic (cigar-shaped) potential, see \cite{ben2005nonlinear} where a similar dimensional reduction was rigorously proved, \cite{bao2003numerical} for some related numerical experiments.
In this case, the rotation term becomes 
\begin{equation*}
\Omega\cdot L=-\ii|\Omega|(x_1\d_2-x_2\d_1).
\end{equation*}
Equation \eqref{eq:gp1} possesses a Hamiltonian given by the total energy 
\begin{equation}\label{eq:en1}
	E[\psi(t)] = \int \frac{1}{2}|\nabla \psi(t, x)|^2 + V(x) |\psi(t, x)|^2 + \frac{g}{2} |\psi(t, x)|^{4}  -(\bar{\psi} (\Omega \cdot L) \psi)(t, x) \, dx.
\end{equation}
Although the GP theory works remarkably well for a broad range of applications, in some specific regimes thermal and quantum fluctuations cannot be ignored. 
In order to accurately describe some specific phenomena (e.g. collective damped oscillations or finite temperature effects), dissipative terms must be taken into account. 
Several models have been proposed to extend the well-established GP description, see \cite{PrJa08} for a general overview. 
For systems close to thermal equilibrium, the dynamics may be described by 
\begin{equation}\label{eq:gp2_intro}
	(\ii - \gamma) \partial_t \psi = -\frac{1}{2}\Delta \psi +V\psi + g|\psi|^{2} \psi - \Omega \cdot L\psi - \mu \psi.
\end{equation}
In this extended model, the damping coefficient $\gamma>0$ is a phenomenological constant that is introduced to fit experimental data, see \cite{BuChMo98} and the references therein.
Moreover, equation \eqref{eq:gp2_intro} is also used to numerically investigate particular features of atomic BECs, such as the formation of vortex lattices \cite{KaTsUe02} or the stabilization of dark solitons \cite{proukakis2004parametric}.
This aspect is actually related to a class of numerical methods \cite{ChSuTo00, Du02, BaDu04} developed to compute ground state solutions to GP-type equations by means of a normalized gradient flow, see \cite{BaoCai} for a general overview and also \cite{AnCaSh22} for the rigorous analysis of the related parabolic model.
\newline
Physical motivations demand this extended model to conserve the total number of particles, that for equation \eqref{eq:gp2_intro} amounts to require the conservation of the $L^2$-norm.
This is achieved by setting the chemical potential $\mu$ to depend on the solution itself as follows
\begin{equation}\label{eq:mu1}
	\mu[\psi(t)] = \frac{1}{\|\psi(t)\|_{L^2}^2} \left( \int  \frac{1}{2} | \nabla \psi(t) |^2	+  V|\psi(t)|^2 + g | \psi(t)|^{4} -(\bar{\psi}(\Omega \cdot L) \psi)(t) dx  \right).
\end{equation}
This choice is also consistent with the numerical methods developed in \cite{ChSuTo00, Du02, BaDu04}, see also \cite{AnCaSh22}.
\newline
Model \eqref{eq:gp2_intro}, with $\mu$ defined by \eqref{eq:mu1}, is also exploited to simulate turbulent phenomena in BECs in the presence of a small damping, see for instance \cite{kobayashi2005kolmogorov}.
\newline
From the mathematical point of view, the chemical potential $\mu[\psi]$ may also be interpreted as a Lagrange multiplier that constrains the solution to live on the manifold
\begin{equation}\label{eq:man}
    \mathcal M=\{u\in L^2(\R^d) \,:\,\|u\|_{L^2}=\|\psi_0\|_{L^2}\}.
\end{equation}
The purpose of our work is to develop a rigorous analysis of the following model
\begin{equation}\label{eq:gp2}
	(\ii - \gamma) \partial_t \psi = -\frac{1}{2}\Delta \psi +V\psi + g|\psi|^{2\sigma} \psi - \Omega \cdot L\psi - \mu \psi,
\end{equation}
with $\mu$ defined in \eqref{eq:mu1} and general power-type nonlinearities satisfying $0<\sigma<\frac{2}{(d-2)^+}$. 
\newline
As it will be outlined in the next subsection, our work focuses on the Cauchy problem, providing local and global well-posedness, depending on the choice of parameters appearing in \eqref{eq:gp2}. Moreover, we investigate the asymptotic behavior of global solutions for large times, showing the convergence towards stationary solutions to \eqref{eq:gp2} 
(and consequently solitary waves to \eqref{eq:gp1} as well). In this respect, our analysis provides a rigorous background for the numerical experiments developed in \cite{KaTsUe02}.
\newline
Let us remark that a similar model was already studied in \cite{ChMaSp15}, where \eqref{eq:gp2} is considered by setting the chemical potential to be constant, i.e. $\mu\in\R$. In this case the total mass is no longer conserved and in fact it may also asymptotically vanish in some cases, see Proposition 4.5. in \cite{ChMaSp15}. In general, the non-conservation of the $L^2$-norm leads to a completely different asymptotic behavior, compare for instance Corollary 6.7. in \cite{ChMaSp15} with Proposition \ref{prp:omega} in the present paper.

\subsection{Mathematical Setting and Main Results}

We are interested in the following Cauchy problem,
\begin{equation}\label{eq:gpe}
	\begin{cases}
         &(\ii - \gamma)\partial_t \psi = -\frac{1}{2}\Delta \psi + V \psi + g|\psi|^{2\sigma} \psi - \Omega \cdot L\psi - \mu[\psi] \psi, \\
        &\psi(0) = \psi_0 \in \Sigma(\R^d),
    \end{cases}
\end{equation}
where $\psi:\R^{1 + d} \to \C$ is complex-valued, 
$d=2, 3$, $g\in\R$ indicates the strength of the nonlinearity and can either be \emph{repulsive} $g>0$ or \emph{attractive} $g<0$, the power $\sigma$ satisfies 
$0< \sigma<\frac{2}{(d-2)^+}$, $\gamma > 0$,  $V$ is defined in \eqref{eq:V} and the angular momentum operator is given in \eqref{eq:ang_mom}. 
The nonlocal, nonlinear term $\mu[\psi]$, which may be interpreted as the chemical potential associated with the state $\psi$, is defined by
\begin{equation}\label{eq:mu_gp}
	\mu[\psi] = \frac{1}{\|\psi\|_{L^2}^2} \left( \int  \frac{1}{2} | \nabla \psi |^2	+  V|\psi|^2 + g | \psi|^{2\sigma + 2} -\bar{\psi}(\Omega \cdot L) \psi dx  \right)
\end{equation}
and ensures the conservation of total mass, namely all solutions to \eqref{eq:gpe} formally satisfy
\begin{equation*}
    \| \psi(t)\|_{L^2} = \|\psi_0\|_{L^2}.
\end{equation*}
We denoted by $\Sigma(\R^d)$ the natural energy space associated with the Cauchy problem \eqref{eq:gpe}, that is defined by
\begin{equation}\label{eq:Sigma1}
	\Sigma(\R^d) = \left\{u \in H^1(\R^d)\, : \, \int |x|^2 |u|^2 dx < \infty  \right\}.
\end{equation} 
Indeed, the total energy associated with \eqref{eq:gpe} is given by
\begin{equation} \label{eq:energy_gpe}
	E[\psi(t)] = \int \frac{1}{2}|\nabla \psi(t)|^2 + V |\psi(t)|^2 + \frac{g}{\sigma + 1} |\psi(t)|^{2\sigma + 2}  -(\bar{\psi} (\Omega \cdot L) \psi)(t)  dx.
\end{equation}
We remark that the rotation term is bounded for all finite energy states, as one has
\begin{equation*}
\left|\int  \bar{u} \Omega\cdot L u \,  dx\right| \lesssim 
\||\cdot| u\|_{L^2}\|\nabla u\|_{L^2}\lesssim
\| u\|_{\Sigma}^2.
\end{equation*}
While the total mass is conserved along the flow of \eqref{eq:gpe}, the total energy is non-increasing and formally satisfies the following identity
\begin{equation}\label{eq:en_decrease}
	E[\psi(t)] + 2\gamma \int_0^t \int |\partial_t \psi(s)|^2 dx ds = E[\psi_0].
\end{equation}
The two main goals of this work are to provide (local and global) well-posedness results for the Cauchy problem \eqref{eq:gpe} and to study the large-time behavior of its global solutions.
\newline
The Hamiltonian equation \eqref{eq:gp1} was already thoroughly studied, see \cite{HaHsLi07, HaHsLi08, AnMaSp10} for local and global well-posedness and the analysis of blow-up solutions.
In \cite{ChMaSp15} model \eqref{eq:gp2} with $\mu\in\R$ was considered. For this dissipative dynamics, it is possible to exploit the parabolic regularization effect, entailed by the linear semi-group. In this way, in \cite{ChMaSp15} the authors prove global well-posedness under some assumptions on the nonlinearity and provide a characterization of the global attractor.
\newline
On the other hand, model \eqref{eq:gpe} has the peculiarity that the chemical potential is determined by a functional of the solution itself, involving its gradient. This actually entails further mathematical difficulties, especially when studying the local well-posedness problem. 
Indeed, the classical contraction argument as developed in \cite{Ka87}, for instance, cannot be straightforwardly adapted to \eqref{eq:gpe}, due to a delicate interplay between the power-type nonlinearity and the nonlocal term. 
In particular, the fact that the nonlinear operator $\psi\mapsto|\psi|^{2\sigma}\psi$ is not locally Lipschitz from $\Sigma$ to itself when $0<\sigma<\frac12$, prevents us to use the standard fixed point argument in this case. We refer to Section \ref{sec:localGP} for a more detailed discussion on this issue.

\begin{theorem} \label{thm:lwp2}
Let  $\frac{1}{2} \leq \sigma < \frac{2}{(d-2)^+}$, or let $g=0$. Then for any  $ \psi_0 \in \Sigma(\R^d)$ there exist a maximal existence time $T_{max}>0$ and a unique local solution $\psi \in C([0,T_{max}), \Sigma(\R^d))$ to equation \eqref{eq:gpe}. Moreover, either $T_{max} = \infty$, or $T_{max} < \infty$ and we have
\begin{equation*}
    \lim_{t \to T_{max}} \| \psi(t)\|_{\Sigma} = \infty.
\end{equation*}
\end{theorem}

Let us notice that this result covers the physically relevant case $\sigma=1$. 
On the other hand, for  the mathematical sake of providing a  more complete picture, we are also going to prove a well-posedness result in the range $0<\sigma<\frac12$.
In order to do this, we develop an alternative argument that constructs a sequence of approximating solutions $\{ \psi^{(k)} \}$. More precisely, we set up an iterative scheme where the chemical potential is given by the previous step, $\mu^{(k)}=\mu[\psi^{(k-1)}]$. In this way, the existence of the approximating solutions at each step is provided by the classical contraction principle.
\newline
Second, we prove that the approximating sequence satisfies uniform bounds which allow us to show its compactness and convergence towards a solution of \eqref{eq:gpe}. 
The drawback is that, in order to infer these uniform bounds, we require the rotation speed to satisfy
$|\Omega| < \frac{\omega}{\sqrt{2}}$,  where $\omega$ is defined by
\begin{equation}\label{eq:min_trap}
\omega=\min_j\omega_j>0.
\end{equation}
Under this assumption, we are able to prove the following theorem.

\begin{theorem} \label{thm:lwp1}
Let  $0 \leq \sigma < \frac{2}{(d-2)^+}$ and 
$|\Omega| < \frac{\omega}{\sqrt{2}}$. 
Then for any  $ \psi_0 \in \Sigma(\R^d)$ there exist a maximal existence time $T_{max}>0$ and a unique local solution $\psi \in C([0,T_{max}), \Sigma(\R^d))$ to equation \eqref{eq:gpe}. Moreover, either $T_{max} = \infty$, or $T_{max} < \infty$ and we have
\begin{equation*}
    \lim_{t \to T_{max}} \| \psi(t)\|_{\Sigma} = \infty.
\end{equation*}
\end{theorem} 

A similar strategy using an iterative scheme has been used also in the context of nonlocal heat equations, see \cite{CaLi09,AnCaSh22}. 
\newline
Let us remark that Theorems \ref{thm:lwp2} and \ref{thm:lwp1} apply to a wider range of nonlinearities with respect to the result in \cite{ChMaSp15}, where the authors assume 
$0<\sigma<\frac{d}{2(d-2)^+}$.
\newline
Let us now address the problem of global well-posedness. By recalling the definition \eqref{eq:min_trap},
we see that the quadratic part of the energy defined in \eqref{eq:energy_gpe} is positive definite if and only if $|\Omega|<\omega$. In particular, for $|\Omega|<\omega$ and for any $u\in\Sigma$ we have
\begin{equation*}
\int\frac12|\nabla u|^2+V|u|^2-\bar u(\Omega\cdot Lu)\,dx
\geq c\int|\nabla u|^2+|x|^2|u|^2\,dx,
\end{equation*}
for some $c>0$.
The case $|\Omega|>\omega$ could lead to instabilities in the dynamics, see for instance \cite{ArNeSp19}, where this issue is studied for the Hamiltonian evolution \eqref{eq:gp1}.

\begin{theorem} \label{thm:gwp}
Let $0 \leq \sigma < \frac12$ with $|\Omega| < \frac{\omega}{\sqrt{2}}$, or
$\frac{1}{2} \leq \sigma < \frac{2}{(d-2)^+}$ with $|\Omega| < \omega$. If $g < 0$, we further assume that $\sigma < \frac{2}{d}$. Then for any  $ \psi_0 \in \Sigma(\R^d)$ there exists a unique global solution $\psi \in C([0,\infty), \Sigma(\R^d))$ to the Cauchy problem \eqref{eq:gpe}. 
\end{theorem}

We now investigate the large-time behavior of global solutions. 
As remarked above, the same question was already addressed in \cite{ChMaSp15} for model \eqref{eq:gp2} in the case when $\mu$ is constant. In this case, the energy dissipation is not sign-definite.
Furthermore, no information can be inferred in general about the asymptotic behavior of the total mass.
This leads the authors of \cite{ChMaSp15} to define the global attractor as the set of all initial data that yield a global solution for all $t\in\R$, i.e. both forward and backward in time. 
\newline
In model \eqref{eq:gpe}, the fact that $\mu=\mu[\psi]$ is given by formula \eqref{eq:mu_gp}, changes substantially the asymptotic behavior of solutions. The energy balance \eqref{eq:en_decrease} heuristically implies that the asymptotic states are given by stationary solutions, whose $L^2$-norm is determined by the initial datum. This property will be rigorously proved in Proposition \ref{prp:omega} below.
\newline
Let us start by discussing equation \eqref{eq:gpe} when $g=0$. We will refer to this setting as the linear case, even though $\mu[\psi] \neq 0$, because it may be seen as the linear evolution constrained on the manifold defined in \eqref{eq:man}.
\newline
In this case, we determine the asymptotic behavior in terms of the decomposition of the initial datum with respect to the eigenspaces of the linear operator
\begin{equation}\label{eq:lin_opGP}
H_\Omega =	- \frac{1}{2} \Delta + V - \Omega \cdot L.
\end{equation}
More precisely, let $\sigma(H_\Omega) = \{\lambda_{n}\}_{n\geq0}$ denote the spectrum of $H_\Omega$ and let $\mathcal W_n$ be the eigenspace associated with $\lambda_n$, namely
\begin{equation*}
\mathcal W_{n}=\{u\in L^2(\R^d)\,:\,H_\Omega u=\lambda_nu\}.
\end{equation*}
Then, if $n^\ast\geq0$ is such that $\lambda_{n^\ast}$ is the smallest eigenvalue in the decomposition of $\psi_0$,
\begin{equation}\label{eq:minEig}
\lambda_{n^\ast} = \inf \left\{ \lambda_{n} : \, \exists\, \varphi \in \mathcal W_{n} : \, (\varphi,\psi_0) \neq 0 \right\},
\end{equation}
we can prove the following result.

\begin{theorem}\label{thm:lin_asym}
Let $g =0$, $\psi_0 \in \Sigma(\R^d)$ and  $\psi \in C([0,\infty),\Sigma(\R^d))$ be the corresponding solution. 
Then 
	\begin{equation*}
		\lim_{t \rightarrow \infty} \mu[\psi(t)] = \lambda_{n^\ast},
	\end{equation*}
 where $\lambda_{n^\ast}$ is defined in \eqref{eq:minEig} and 
	\begin{equation*}
		\lim_{t \rightarrow \infty}\inf_{u \in \mathcal W_{n^\ast}}\| \psi(t) - u\|_{\Sigma} = 0.  
	\end{equation*}
\end{theorem}

In particular, if $\lambda_{n^\ast}$ is a simple eigenvalue,  then there exists an eigenfunction $\psi_{n^\ast} \in \mathcal {W}_{n^\ast}$, with 
$\|\psi_{n^\ast}\|_{L^2} = \| \psi_0\|_{L^2}$ 
such that 
\begin{equation*}
\lim_{t \rightarrow \infty}\| \psi(t)-\psi_{n^\ast}\|_{\Sigma} = 0.  
\end{equation*}
This is for example the case for the smallest eigenvalue $\lambda_0$ in the spectrum of $H_\Omega$. 
Thus, any initial condition having a non-zero component in $\mathcal W_{\lambda_0}$ evolves into a solution that asymptotically converges to the least energy state, with a given $L^2$-norm.
\newline
In general, the linear dynamics leaves invariant the eigenspaces $\mathcal W_n$ associated with the operator $H_\Omega$.
On the contrary, the power-type nonlinearity destroys this property, so that for $g\neq0$ there are no linear invariant subspaces anymore.
\newline
In the nonlinear case, we expect that stationary solutions play the same role as eigenfunctions of $H_\Omega$ when $g=0$. We recall that a stationary solution to \eqref{eq:gpe} satisfies  the following equation
\begin{equation}\label{eq:sol_wa}
0 =  -\frac{1}{2}\Delta Q + VQ + g|Q|^{2\sigma} Q - \Omega \cdot LQ - \mu[Q] Q.
\end{equation}
It is well known that, under the same assumptions of Theorem \ref{thm:gwp}, there are infinitely many solutions to \eqref{eq:sol_wa}. 
Among stationary solutions, a particular role is played by minimizers of the energy with fixed total mass,
 \begin{equation}\label{eq:min_pb}
\tau = \inf_{u \in \Sigma} \{ E[u]\, : \, \| u\|_{L^2} =m\}.
\end{equation}
Under the assumptions of Theorem \ref{thm:gwp}, a minimizer to problem \eqref{eq:min_pb} always exists \cite{Se02, ArNeSp19}.
In what follows we will refer to them as ground states.
\newline
By computing the Euler-Lagrange equations and by exploiting Pohozaev's identity, it is straightforward to show that any ground state satisfies \eqref{eq:sol_wa}.
Moreover, we are going to call excited states all stationary solutions, whose total energy is strictly larger than the ground state energy.
Given $m>0$, the set of all stationary solutions with total mass $m$, is denoted by
\begin{equation}
	\mathcal{S}_m = \{u \in \Sigma(\R^d): \,  \| u\|_{L^2}=m, \, u \mbox{ solves } \eqref{eq:sol_wa} \}. \label{eq:S}
\end{equation}
In general, the monotonicity of the total energy, inferred from \eqref{eq:en_decrease}, implies that
\begin{equation}\label{eq:asy_en}
\lim_{t\to\infty}E[\psi(t)]=E_\infty\geq\tau>-\infty,
\end{equation}
where $\tau$ is defined in \eqref{eq:min_pb}. Furthermore, from the energy dissipation in \eqref{eq:en_decrease}, we also expect that the asymptotic states for \eqref{eq:gpe} are determined by stationary solutions. 
\begin{theorem}\label{thm:subconv}
Under the conditions stated in Theorem \ref{thm:gwp}, let $\psi_0 \in \Sigma(\R^d)$ be such that $\|\psi_0\|_{L^2}=m$ and let $\psi \in C([0,\infty),\Sigma(\R^d))$ be the corresponding solution to \eqref{eq:gpe}. Then there exists a sequence $\{t_n\} \subset \R^+$, $t_n \rightarrow \infty$ as $n \rightarrow \infty$ and a stationary state $Q \in \mathcal{S}_m$  such that 
	\begin{equation*}
 \lim_{n\to\infty}\|\psi(t_n)-Q\|_{\Sigma}=0,
\end{equation*}
and
	\begin{equation*}
E_\infty= E[Q],
	\end{equation*}
where $E_\infty$ is defined in \eqref{eq:asy_en}.
\end{theorem}
Let us remark that, while it is always possible to show that $E_\infty=E[Q]$, for some $Q\in\mathcal S_m$, in general we are not able to characterize the stationary solution that determines the asymptotic behavior of the dynamics. In fact, it is not always true that $E_\infty=\tau$, as all stationary states are solutions to \eqref{eq:gp2}. On the other hand, the monotonicity of the total energy suggests that all excited states are unstable under the dynamics. If the initial energy is strictly smaller than the energy of any excited state, then we expect the solution to converge towards the ground state. These arguments lead to exploit the following conjecture.

\begin{conjecture}[Fundamental gap conjecture]\label{concon}  
	There exists $\delta >0$ such that for any $Q \in \mathcal{S}_m$, either $Q = e^{\ii\phi}Q_{gs}$ for some $\phi \in \R$, or $E[Q] \geq E[Q_{gs}] + \delta$. 
\end{conjecture}
The fundamental gap conjecture is largely studied in the mathematical literature, due to its relevance.
There are several numerical works and formal expansions  confirming this conjecture, see for instance \cite{cances2010numerical, BaRu18} and references therein, although a theoretical proof seems to be still missing. 
Under the hypothesis that this conjecture is true, we are able to improve our convergence result in two directions. First, we show that, up to phase shifts, the dynamics asymptotically converges to the ground state. Second, no extraction of a sequence of times is needed. This is achieved by adapting the convexity argument introduced by Cazenave and Lions \cite{CaLi82}. 

\begin{theorem}\label{thm:conv}
		Assume that Conjecture \ref{concon} is true. Under the conditions stated in Theorem \ref{thm:gwp}, let $\psi_0 \in \Sigma(\R^d)$ be such that $E[\psi_0] < E[Q_{gs}] + \delta$ and  $\psi \in C([0,\infty),\Sigma(\R^d))$ be the corresponding  solution to \eqref{eq:gpe}. Then there exists $\phi \in [0,2\pi)$ such that
	\begin{equation*}
 \lim_{t\to\infty}\|\psi(t)-e^{\ii \phi} Q_{gs}\|_{\Sigma}=0.
	\end{equation*}
\end{theorem} 

The theorem above states that solutions starting from an initial condition that is sufficiently close to the ground state eventually converge to it; consequently the ground state is asymptotically stable, up to a phase shift. 
However, no smallness condition is required, as opposed to the Hamiltonian dynamics.
\newline
We recall that the orbital stability of the ground state for the non-dissipative rotational GP equation \eqref{eq:gp1} was proved in \cite{ArNeSp19}. For qualitative properties of the ground state such as symmetry breaking, we refer the reader to \cite{IgMi06,Se02,Se03}. An alternative approach to study solitary waves and their stability for equation \eqref{eq:gp1} with super-quadratic trapping potentials was recently investigated in \cite{nenciu2023nonlinear}.
\newline
This work is organized as follows: in Section \ref{sec:prelGP}, we will present our notations and some preliminary results. 
Section \ref{sec:localGP} is dedicated to proving the local and global well-posedness results stated in Theorems \ref{thm:lwp2}, \ref{thm:lwp1} and \ref{thm:gwp}. 
In Section \ref{sec:linASGP} we will study the asymptotic behavior of the linear case $g = 0$. Finally, in Section \ref{sec:asNNGP}, we prove our result on the asymptotic behavior in  the nonlinear case.

\section{Preliminaries}\label{sec:prelGP}

In this section, we introduce our notations and recall some mathematical tools. 
For two positive quantities $A$ and $B$ we use the notation
\begin{equation*}
    A \lesssim B
\end{equation*}
with the meaning that there exists a constant $c >0$, not depending on  $A$ and $B$, unless specifically declared, such that
\begin{equation*}
    A \leq cB.
\end{equation*}
We will use the following notation
\begin{equation*}
	\frac{1}{(d-2)^+} = \begin{cases}
		& \infty \ \mbox{ for } \ d=2, \\
		& \frac{1}{d-2} \ \mbox{ for } \ d = 3.
	\end{cases}
\end{equation*}
Given $u,v \in L^2(\R^d)$, the $L^2$-scalar product is defined as
\begin{equation*}
    (u,v) = Re \int u \bar{v} dx.
\end{equation*}
We define the set
\begin{equation*}
		\Sigma^2(\R^d) = \{ u \in \Sigma(\R^d) \,: \, \|(- \Delta + |x|^2) u \|_{L^2} < \infty \} \subset H^2(\R^d).
\end{equation*}
where $\Sigma(\R^d)$ is defined in \eqref{eq:Sigma1}. 
By abuse of notation, we write $xu$ for the function $x \to xu(x)$. We recall some basic facts about the space $\Sigma(\R^d)$, see for example, \cite{KaWe94}, for more details.

\begin{prop}\label{prp:comemb}
 The Hilbert space $\Sigma(\R^d)$ is compactly embedded in $L^p(\R^d)$ for any $p\in \left[2,\frac{2d}{(d-2)^+} \right).$ Moreover, the norm
 \begin{equation*}
 \| u\|_\Sigma^2 = \| \nabla u \|_{L^2}^2 + \| xu\|_{L^2}^2
 \end{equation*}
 is equivalent to the norm induced by the scalar product
 \begin{equation*}
 \langle u,v \rangle_\Sigma = (u,v) + (\nabla u,\nabla v) + (xu,xv).
 \end{equation*}
\end{prop}
	
\begin{proof}
 For any $R>0$, we have that
 \begin{equation*}
 \int_{|x| \geq R} |u(x)|^2\, dx \leq R^{-2} \| xu\|_{L^2}^2.
 \end{equation*}
On the other hand, we also know that $H^1(B_R)$ compactly embeds into $L^2(B_R)$, where $B_R \equiv \{x \in \R^d \, :\, |x|<R\}$. Consequently, we may conclude that the embedding $\Sigma (\R^d)\hookrightarrow L^2(\R^d)$ is compact. 
The embedding $H^1(\R^d) \hookrightarrow L^p(\R^d)$ for $p \in \left[ 2, \frac{2d}{(d-2)^+}\right)$, implies that $\Sigma (\R^d)\hookrightarrow L^p(\R^d)$. Choosing any $q>p$, such that $q < \frac{2d}{(d-2)^+}$ and interpolating between $L^2(\R^d)$ and $L^q(\R^d)$ implies that the embedding $\Sigma(\R^d) \hookrightarrow L^p(\R^d)$ is compact. 
\end{proof}	
	
\begin{prop}\label{prp:sigemb}
 The Hilbert space $\Sigma^2(\R^d)$ is compactly embedded in $\Sigma(\R^d)$.
\end{prop}

\begin{proof}
 Let $u\in\Sigma^2(\R^d)$ and notice that the condition 
 \begin{equation*}
 \|(- \Delta + |x|^2) u\|_{L^2}^2 < \infty, 
 \end{equation*}
 is equivalent to 
 \begin{equation*}
		 \| \Delta u\|_{L^2} + \| |x|^2 u\|_{L^2} + \big\| |x| |\nabla u| \big\|_{L^2} < \infty.
	\end{equation*}
	Thus we define the norm of $\Sigma^2(\R^d)$ as
		\begin{equation}\label{eq:normSig2}
		 \| u\|_{\Sigma^2} = \left(\| \Delta u\|_{L^2}^2 + \| |x|^2 u\|_{L^2}^2 + \big\| |x| |\nabla u| \big \|_{L^2}^2 \right)^\frac{1}{2}.
	\end{equation}
	Consequently, adapting the proof of Proposition \ref{prp:comemb}, one can prove that the embedding $\Sigma^2(\R^d) \hookrightarrow \Sigma(\R^d)$ is compact. 
\end{proof}	

Standard compactness arguments then yield the following existence result for ground states of \eqref{eq:energy_gpe}, see \cite{Se02}.

\begin{prop}\label{prp:gsGP}
Let $|\Omega| < \omega$, $g \geq 0$ and $\sigma < \frac{2}{(d-2)^+}$ or $g < 0$ and $\sigma < \frac{2}{d}$. For any $m >0$, there exists $Q_m \in \Sigma(\R^d)$ solving the following variational problem:
 \begin{equation*}
 \tau_m = \inf \{ E[u]: \, u \in \Sigma, \ \| u\|_{L^2} = m \}.
 \end{equation*}
 Moreover, if $v\in \Sigma(\R^d)$ satisfies $E[v] = \tau_m$ and $\| v\|_{L^2} = m$ then there exists $\varphi \in [0,2\pi)$ such that
 \begin{equation*}
 v = e^{\ii \varphi} Q_m.
 \end{equation*}
\end{prop}
For $\Omega = 0$, there exists a unique real-valued and strictly positive minimizer, which is usually called the ground state. Observe that when $\Omega \neq 0$, all the minimizers are complex-valued for the presence of the rotational operator $L$. Consequently, we refer to all the minimizers of the problem above as ground states. In general, the properties of stationary states, such as radial symmetry breaking can be found in \cite{Se02}, \cite{Se03}.

\subsection{Properties of the semigroup}

We begin with the presentation of the smoothing estimates of the linear Hamiltonian operator
\begin{equation*}
	H_\Omega =	- \frac{1}{2} \Delta + V - \Omega \cdot L,
\end{equation*}
where $V$ is defined in \eqref{eq:V}.
 We denote by 
\begin{equation}\label{eq:propagGP}
U_\Omega (t) = e^{-\frac{i+\gamma}{1+\gamma^2} tH_\Omega}
\end{equation} 
the semi-group associated with $H_\Omega$ on $L^2 (\R^d)$, so that, given $u_0 \in \Sigma (\R^d)$, the function $U_\Omega (t) u_0$ solves the linear problem
\begin{equation} \label{eq:linGPeq}
\begin{cases}
 \partial_t u = -\frac{i+\gamma}{1+\gamma^2} H_\Omega u, \ \mbox{in} \ (0,\infty) \times \R^d, \\
 u(0) = u_0.
\end{cases}
\end{equation}
We will use the following commutators estimates.

\begin{prop}\label{prp:comm}
We have the following properties:
\begin{equation}\label{eq:commEst}
	[\nabla,H_\Omega] = \nabla V + \ii \nabla \wedge \Omega, \quad 	[x,H_\Omega] = \nabla - \ii \Omega\wedge x.
\end{equation}
\end{prop}

\begin{proof}
 We compute explicitly the commutator $	\left[\nabla,H_\Omega\right]$ as 
 \begin{equation*}
 \begin{aligned}
 \left[\nabla,H_\Omega\right] &= -\frac{1}{2} \left[\nabla,\Delta\right] + \left[\nabla,V\right] + \ii \left[\nabla,\Omega \cdot (x\wedge \nabla) \right] \\ 
 &= \nabla V + \ii \left[\nabla,x\cdot (\nabla \wedge \Omega) \right] \\
 & = \nabla V + \ii \nabla \wedge \Omega,
 \end{aligned}
 \end{equation*}
 where we used the property $a \cdot(b \wedge c) = det(a,b,c) = (a \wedge b) \cdot c$ for three dimensional vectors $a,b,c$. Similar calculations yield 
 \begin{equation*}
 [x,H_\Omega] = \nabla - \ii \Omega\wedge x.
 \end{equation*}
\end{proof}

We recall the dispersive estimates of this semi-group.

\begin{prop}[Lemma 2.3 in \cite{ChMaSp15}]
There exists a $t^* >0$ such that, for all $1 \leq q \leq p \leq \infty$ and $t \in [0,t^*)$,
	\begin{equation}\label{eq:disp1}
		\| U_\Omega (t) f\|_{L^p} \lesssim t^{\frac{d}{2} \left( \frac{1}{p} - \frac{1}{q} \right)} \| f\|_{L^q},
	\end{equation}
and 
\begin{equation}\label{eq:disp2}
	\|\nabla U_\Omega (t) f\|_{L^p} + \|x U_\Omega (t) f\|_{L^p} \lesssim t^{\frac{d}{2} \left( \frac{1}{p} - \frac{1}{q} \right) - \frac{1}{2}} \| f\|_{L^q}.
\end{equation}
\end{prop} 

The proof is based on the Mehler formula \cite{Ca09} and a change of variables introduced in \cite{AnMaSp10} which transforms the linear operator $H_\Omega$ into a different linear operator without rotation but with a time-dependent harmonic potential, see \cite[Lemma 2.3]{ChMaSp15} for more details. 
\newline
We notice the following property of the semigroup.

\begin{prop}
 Let $u_0 \in \Sigma (\R^d)$ and let $U_\Omega (t) u_0 \in C\big([0,\infty), \Sigma (\R^d) \big)$ be the corresponding solution to \eqref{eq:linGPeq}. Then for any $t \in (0,\infty)$ we have
 \begin{equation}\label{eq:GPlinL2}
 \big\| U_\Omega (t) u_0 \big\|_{L^2}^2 = \big\| u_0 \big\|_{L^2}^2 - \frac{2 \gamma}{1 + \gamma^2} \int_0^t \big( H_\Omega (U_\Omega (\tau) u_0),U_\Omega (\tau) u_0 \big) \, d\tau.
 \end{equation}
 Moreover there exists $C >0$ such that
 \begin{equation}\label{eq:GPlinL3}
 \big\| U_\Omega (t) u_0 \big\|_{L^2}^2 \geq \big\| u_0 \big\|_{L^2}^2 - C t \| u_0\|_{L^\infty([0,t],\Sigma)}^2.
 \end{equation}
\end{prop}

\begin{proof}
 We rewrite system \eqref{eq:linGPeq} as 
 \begin{equation}\label{eq:linGPeq1}
 \begin{cases}
 (\ii - \gamma) \partial_t u = H_\Omega u, \\
 u(0) = u_0 \in \Sigma(\R^d).
 \end{cases}
 \end{equation}
 By taking the scalar product of \eqref{eq:linGPeq1} with $u$, we get
 \begin{equation}\label{eq:linGP1}
 (\ii \partial_t u, u) - \frac{\gamma}{2} \frac{d}{dt} \| u(t) \|_{L^2}^2 = \big(H_\Omega u, u \big). 
 \end{equation}
 Moreover, by taking the scalar product of \eqref{eq:linGPeq1} with $\ii u$ we also have that
 \begin{equation*}
 \frac{1}{2} \frac{d}{dt} \| u(t) \|_{L^2}^2 - \gamma (\partial_t u, \ii u) = (H_\Omega u, \ii u) = 0,
 \end{equation*}
 that is 
 \begin{equation}\label{eq:linGP2}
 (\ii \partial_t u, u) = - \frac{1}{2 \gamma} \frac{d}{dt} \| u(t) \|_{L^2}^2.
 \end{equation}
 Plugging \eqref{eq:linGP2} inside \eqref{eq:linGP1} leads to obtain that
 \begin{equation*}
 \frac{d}{dt} \| u(t) \|_{L^2}^2 = - \frac{2 \gamma}{1 + \gamma^2} \big( H_\Omega u, u \big). 
 \end{equation*}
 Equation \eqref{eq:GPlinL2} follows from this equality after integration in time. Furthermore, \eqref{eq:GPlinL3} follows from \eqref{eq:GPlinL2} after employing the commutator estimates \eqref{eq:commEst} and the dispersive estimate \eqref{eq:disp1}.
\end{proof}

We proceed with the following definition. 

\begin{defin}\label{def:adm}
	 A pair $(q, r)$ is called admissible if $q \geq 2$, $(q,r,d) \neq (2,\infty,2)$, and 
	 \begin{equation}\label{eq:qrAdm}
	 	\frac{2}{q} = d \left( \frac{1}{2} - \frac{1}{r} \right). 
	 \end{equation}
\end{defin}

Having the dispersive estimates \eqref{eq:disp1} and \eqref{eq:disp2}, it is possible to adapt the standard arguments to prove Strichartz-type estimates for the linear propagator $U_\Omega(t)$, see for instance Proposition 2.12 in \cite{AnMiSc18} where similar Strichartz-type estimates are proven in the case without potential, without rotation.

\begin{prop}\label{prp:strich}
	Let $(q,r)$ and $(s,p)$ be two Strichartz admissible pairs and let $T>0$. Then we have
		\begin{equation*}
		\| U_\Omega(t) \varphi\|_{L^q([0,T],L^r)} \lesssim \| \varphi\|_{L^2}.
	\end{equation*}
 Moreover, if 
	\begin{equation*}
		F(\varphi) = \int_0^t U_\Omega(t-\tau) \varphi(\tau,x) d\tau,
	\end{equation*}
then we also have that
	\begin{equation*}
		\| F(\varphi)\|_{L^q([0,T],L^{r})} \lesssim\|\varphi\|_{L^{s'}([0,T],L^{p'})}
	\end{equation*}	
and
 \begin{equation*}
 	\|\nabla F(\varphi)\|_{L^q([0,T],L^r)} + \|x F(\varphi)\|_{L^q([0,T],L^r)} \lesssim \| \nabla \varphi\|_{L^{s'}([0,T],L^{p'})} + \| x\varphi\|_{L^{s'}([0,T],L^{p'})}.
 \end{equation*}	
\end{prop}

\subsection{Spectral properties of the linear Hamiltonian}

For $\Omega = 0$, the Hamiltonian defined in \eqref{eq:lin_opGP} is the anisotropic quantum mechanical oscillator 
\begin{equation*}
 H_0 = \frac{1}{2} \left( - \Delta + \sum_{j=1}^d \omega_j^2 x_j^2 \right).
\end{equation*}
The spectral properties of this operator are well-known, see for example \cite{Te09}.
\begin{prop}
 $H_0$ is essentially self-adjoint on $C_0^\infty (\R^d)\subset L^2(\R^d)$ with compact resolvent. If for all $j$, $\omega_j = \omega$, then the spectrum $\sigma(H_0) = \{	\lambda_{0,n}\}$ is given by 
	\begin{equation*}
		\lambda_{0,n} = \omega \left(\frac{d}{2} + n - 1 \right),
	\end{equation*}
and the eigenvalue $	\lambda_{0,n}$ is 
\begin{equation*}
 \left( \begin{matrix} d + n -2 \\ n-1
 \end{matrix}
 \right) \mbox{-fold degenerate.}
\end{equation*}
\end{prop}
In particular, the smallest eigenvalue is given by $\lambda_{0,1} = \frac{\omega d}{2} > 0$. Notice that the associated eigenfunctions form a complete orthonormal basis of $L^2(\R^d)$. 
\newline
In the case $\Omega \neq 0$, we have the following \cite{ChMaSp15}.
\begin{prop}\label{prp:eigenv}
	Consider $ \omega > |\Omega|$. Then	$H_\Omega$ is essentially self-adjoint on $C_0^\infty (\R^d) \subset L^2 (\R^d)$, with compact resolvent and discrete spectrum. Moreover, if for any $j$, $\omega_j = \omega$, then the spectrum of $H_\Omega$ is given by $\sigma(H_\Omega) = \{ \lambda_{\Omega,n} \}_{n \in \N} $ where
	\begin{equation*}
	\{ \lambda_{\Omega,n} \}_{n \in \N} = \{ \lambda_{0,k} + m \Omega, \ -k + 1 \leq m \leq k - 1, \, \mbox{for} \ k \in \N \}.
	\end{equation*}
\end{prop}

In particular, when $|\Omega| < \omega$, we still have that the smallest eigenvalue is $\lambda_{\Omega,1} = \frac{\omega d}{2} > 0$. Consequently, ground state's energy remains the same with rotation.

\section{Existence of Solutions} \label{sec:localGP}

In this section, we study local and global well-posedness properties of equation \eqref{eq:gpe}.
\newline
As already discussed in the introduction, we will present two different strategies to study local solutions.
The standard fixed point argument for evolutionary problems such as \eqref{eq:gpe}, see for instance \cite{Ka87}, consists in finding a complete metric space and a nonlinear contraction whose unique fixed point provides us with the solution.
\newline
The nonlinear map is usually associated with the integral formulation of the equation, which in our case reads
\begin{equation*}
\psi(t)=U_\Omega(t)\psi_0
-\frac{i+\gamma}{1+\gamma^2}\int_0^tU_\Omega(t-s)\left(
g|\psi|^{2\sigma}\psi-\mu[\psi]\psi\right)(s)\,ds.
\end{equation*}
The presence of the nonlocal term $\mu[\psi]$ prevents us to follow the approach in \cite{Ka87} and defining the metric space endowed with the weaker distance based on mixed space-time Lebesgue norms. In particular, the distance must also control the gradient of the difference between the two functions. On the other hand, defining a distance based on Sobolev spaces requires the nonlinear term to be locally Lipschitz in those spaces, for instance from $H^1$ into itself. This implies that we can exploit the standard fixed point argument only when $\frac12\geq\sigma$. 
This strategy will be discussed in subsection \ref{sec:lwp2}.
\newline
To cover also the case $0<\sigma<\frac12$, we developed a different argument, based on an iterative procedure. We construct a sequence of approximating solutions $\{\psi^{(k)}\}$, where at each step the chemical potential is determined by the previous iterate. More precisely, for any $k\geq1$, we study the following Cauchy problem
\begin{equation*}
\begin{cases}
 (\ii - \gamma)\partial_t \psi^{(k)} = -\frac{1}{2}\Delta \psi^{(k)} +V\psi^{(k)} + g|\psi^{(k)}|^{2\sigma} \psi^{(k)} - \Omega \cdot L\psi^{(k)} - \mu[\psi^{(k-1)}] \psi^{(k)}, \\
 \psi^{(k)}(0) = \psi_0 \in \Sigma(\R^d). 
\end{cases} 
\end{equation*}
This allows us to exploit the standard fixed point argument at each step, with no restriction on $\sigma$. To find a solution to the original equation, we will prove uniform bounds on $\{\psi^{(k)}\}$ and pass to the limit. Inferring these priori estimates on the $L^2([0, T), \Sigma^2(\R^d))$-norm of the approximating profiles requires to assume that $|\Omega|<\frac{\omega}{\sqrt2}$.
This strategy will be discussed in subsection \ref{sec:secondLocal}.

\subsection{Local well-posedness for $\sigma \geq \frac{1}{2}$}\label{sec:lwp2}

In this subsection, we present our first proof for the local well-posedness result, which is based on a fixed point argument and requires to suppose that $\sigma \geq \frac{1}{2}$. We will now prove Theorem \ref{thm:lwp2}.

\begin{proof}[Proof of Theorem \ref{thm:lwp2}:]
We will show the proof for $g\neq0$, the other case being a straightforward adaptation. Fix $M, N, T>0$, to be chosen later, and let 
\begin{equation}\label{eq:rQThetaGp}
 r=2\sigma +2, \quad q = \frac{4 \sigma + 4}{d \sigma}. 
\end{equation}
Notice that the pair $(q, r)$ satisfies condition \eqref{eq:qrAdm}. Consider the set
\begin{equation*}
	\begin{aligned}
		 \mathcal E = \bigg\{ & u \in L^{\infty}\left([0,T], \Sigma(\R^d) \right), \ u,xu,\nabla u\in L^{q}\left( (0,T), L^{r}(\R^d)\right): \, \\ 
       & \| u \|_{L^{\infty}\left([0,T], \Sigma(\R^d)\right)} \leq M, 
		  \| u \|_{L^{q}\left((0,T), W^{1,r} \right)} \leq M, \ 
        \| x u \|_{L^{q}\left((0,T), L^r\right)} \leq M, \\ 
        &\inf_{t \in [0,T]} \| u \|_{L^2} \geq \frac{N}{2} \bigg\},		 
	\end{aligned}
\end{equation*}
equipped with the distance
\begin{equation*}
	\mathrm{d}(u, v)=\|u-v\|_{L^{q}\left((0, T), W^{1,r}\right)}+\|u-v\|_{L^{\infty}\left([0, T], \Sigma \right)} + \|x(u-v)\|_{L^{q}\left((0, T), L^r\right)}.
\end{equation*}
Clearly $(\mathcal E, d)$ is a complete metric space. We fix $u,v \in \mathcal E$. We observe that the inequality 
\begin{equation*}
	\left||u|^{2\sigma}u - |v|^{2\sigma} v \right| \leq C\left(|u|^{2\sigma}+|v|^{2\sigma}\right)|u-v|,
\end{equation*} 
implies that
\begin{equation*}
	\begin{aligned}
		\left\||u|^{2\sigma} u - |v|^{2\sigma} v \right\|_{L^{q^{\prime}}\left((0, T), L^{r^\prime}\right)} & \lesssim 
		\left(\|u\|_{L^{\infty}\left([0, T], L^{r}\right)}^{2\sigma}+\|v\|_{L^{\infty}\left([0,T], L^{r}\right)}^{2\sigma}\right) \\ 
        &\times \|u-v\|_{L^{q^\prime}\left((0, T), L^{r}\right)}.
	\end{aligned}
\end{equation*}
Moreover, using the embedding $\Sigma\left(\R^{d}\right) \hookrightarrow L^{r}\left( \R^d \right)$ and H\"older's inequality, it also follows that
\begin{equation*}
	\left\| \nabla \left( |u|^{2\sigma} u - |v|^{2\sigma} v \right) \right\|_{L^{r^{\prime}}} \lesssim \left( \|u\|_{L^{r}}^{2\sigma-1} + \|v\|_{L^{r}}^{2\sigma-1}\right) \left( \|\nabla u\|_{L^{r}} + \|\nabla v\|_{L^{r}} \right) \|u - v\|_{H^1}
\end{equation*}
which implies
\begin{equation}\label{eq:omega2GP}
	\begin{aligned}
		\left\|\nabla \left(|u|^{2\sigma} u - |v|^{2\sigma} v \right) \right\|_{L^{q^{\prime}}\left((0, T), L^{r^{\prime}}\right)} &\lesssim \left(\|u\|_{L^{\infty}\left([0,T], L^{r}\right)}^{2\sigma-1}+\|v\|_{L^{\infty}\left([0,T], L^{r}\right)}^{2\sigma-1}\right) \\ 
		&\times \left( \| \nabla u\|_{L^{q^\prime}((0, T), L^{r})} + \| \nabla v\|_{L^{q^\prime}((0, T), L^{r})} \right)\\
        &\times \|u-v\|_{L^{\infty}\left([0,T], H^1\right)}.
	\end{aligned} 
\end{equation}
In the same way, from 
\begin{equation*}
	\left\| x\left( |u|^{2\sigma} u - |v|^{2\sigma} v \right) \right\|_{L^{r^{\prime}}} \lesssim \left( \|u\|_{L^{r}}^{2\sigma-1} + \|v\|_{L^{r}}^{2\sigma-1}\right) \left( \|x u\|_{L^{r}} + \|x v\|_{L^{r}} \right) \|u - v\|_{H^1}
\end{equation*}
we also obtain
\begin{equation}\label{eq:omega2GP1}
	\begin{aligned}
		\left\|x \left(|u|^{2\sigma} u - |v|^{2\sigma} v \right) \right\|_{L^{q^{\prime}}\left((0, T), L^{r^{\prime}}\right)} & \lesssim \left(\|u\|_{L^{\infty}\left([0,T], L^{r}\right)}^{2\sigma-1}+\|v\|_{L^{\infty}\left([0,T], L^{r}\right)}^{2\sigma-1}\right) \\ 
		& \times \left( \| x u\|_{L^{q^\prime}((0, T), L^{r})} + \|xv\|_{L^{q^\prime}((0, T), L^{r})} \right)\\ 
        & \times \|u-v\|_{L^{\infty}\left([0,T], H^1\right)}.
	\end{aligned} 
\end{equation}
Using H\"older's inequality in time, we deduce from the above estimates that
\begin{equation*}
	 \left\| x |u|^{2\sigma}u \right\|_{L^{q^{\prime}}\left((0, T), L^{r^{\prime}}\right)} + \left\||u|^{2\sigma}u \right\|_{L^{q^{\prime}}\left((0, T), W^{1, r^{\prime}}\right)} \lesssim \left(T+T^{\frac{q-q^{\prime}}{q q^{\prime}}}\right)\left(1+M^{2\sigma}\right) M
\end{equation*}
 and
\begin{equation*}
\begin{aligned}
 &\left\|x(|u|^{2\sigma}u - |v|^{2\sigma} v) \right\|_{L^{q^{\prime}}\left((0, T), L^{r^{\prime}}\right)} + \left\||u|^{2\sigma}u - |v|^{2\sigma} v \right\|_{L^{q^{\prime}}\left((0, T), W^{1, r^{\prime}}\right)} \\
	&\lesssim \left(T+T^{\frac{q-q^{\prime}}{q q^{\prime}}}\right)\left(1+M^{2\sigma}\right) \mathrm{d}(u, v) .
\end{aligned}
\end{equation*}
Next, given any $u, v \in \mathcal E$, we notice that we have
\begin{equation*}
	\left| \mu[u] u - \mu[v] v \right| \leq |v| \left| \mu[u] - \mu[v] \right| + \left| \mu [u] \right| \left| u - v\right|
\end{equation*}
which implies
\begin{equation*}
	\begin{aligned}
	\left\| \mu[u]u - \mu[v] v \right\|_{L^{1}\left((0, T),\Sigma \right)} \lesssim T\bigg(& \| \mu[u]\|_{L^\infty[0,T]} \| u - v\|_{L^{\infty}\left([0,T],\Sigma \right)} \\ 
	 & +\| v\|_{L^{\infty}\left([0,T],\Sigma \right)} \left\| \mu[u] - \mu[v] \right \|_{L^{\infty}[0,T]} \bigg).
	\end{aligned}
\end{equation*}
By the definition of $L$, see \eqref{eq:ang_mom}, we notice that for any $u \in \Sigma$, from the Cauchy-Schwarz inequality we have
\begin{equation}\label{eq:rotenCN}
 \left| \int \Omega \cdot Lu \, \bar{u} \, dx \right| \leq |\Omega| \| x u \|_{L^2}\|\nabla u \|_{L^2} \lesssim \| u \|_{\Sigma}^2.
\end{equation}
Thus, for any $u \in \mathcal E$, there exists $C(M,N) >0$ such that
\begin{equation}\label{eq:omegaLGP}
 \| \mu[u]\|_{L^\infty[0,T]} \lesssim \frac{1}{N^2} \big( \| u\|_{L^\infty([0,T],\Sigma)}^2 + \| u\|_{L^\infty([0,T],\Sigma)}^{2\sigma + 2}\big) \leq C.
\end{equation}
Moreover, from the embedding $\Sigma(\R^d) \hookrightarrow L^{r}\left( \R^d \right)$ we also obtain that 
\begin{equation}\label{eq:omega1GP}
\begin{aligned}
 \left\| \mu[u] - \mu[v] \right\|_{L^{\infty}[0,T]} \leq C \| u- v \|_{L^{\infty}\left([0,T],\Sigma \right)}.
\end{aligned}
\end{equation}
For any $u_0 \in \Sigma(\R^d)$, let $\mathcal{H}(u_0)(u)(t) = \mathcal{H}(u)(t) $ be defined as
\begin{equation*}
	\mathcal{H}(u)(t)= U_\Omega(t) u_0 - \frac{i+\gamma}{1+\gamma^2} \int_0^t U_\Omega(t - \tau)(g |u(\tau)|^{2\sigma}u(\tau) - \mu[u(\tau)] u(\tau)) \, d\tau,
\end{equation*}
where $U_\Omega(t)$ is defined in \eqref{eq:propagGP}.
By using Proposition \ref{prp:strich}, the embedding $\Sigma(\R^d) \hookrightarrow L^r(\R^d) $ and H\"older's inequality, we obtain
\begin{equation}\label{eq:str11GP}
	\begin{split}
	\| \mathcal H(u)\|_{L^q([0,T),L^r) \cap L^\infty([0,T],L^2)} &\lesssim \| u_0\|_{L^2} + T^{\frac{q-q^{\prime}}{q q^{\prime}}} \| u\|^{2\sigma}_{L^\infty([0,T],\Sigma)} \| u\|_{L^q([0,T),L^r)} \\ 
	&+ T \| \mu[u] \|_{L^\infty[0,T]} \| u\|_{L^\infty([0,T],L^2)}.
	\end{split}
\end{equation}
Moreover, using the commutator estimates in Proposition \ref{prp:comm}, we get that
\begin{equation*}
	\begin{split}
		\nabla \mathcal H(u) &= U_\Omega(t) \nabla u_0 - \frac{i+\gamma}{1+\gamma^2} \int_0^t U_\Omega(t - \tau)\nabla(g |u(\tau)|^{2\sigma}u(\tau) - \mu[u(\tau)] u(\tau)) \, d\tau \\
		& - \frac{i+\gamma}{1+\gamma^2} \int_0^t U_\Omega(t - \tau)(\nabla V -\ii\Omega \wedge \nabla) \mathcal H(u) (\tau) \, d\tau 
	\end{split}
\end{equation*}
and
\begin{equation*}
	\begin{split}
	x \mathcal H(u) &= U_\Omega(t) x u_0 - \frac{i+\gamma}{1+\gamma^2} \int_0^t U_\Omega(t - \tau)\nabla(g |u(\tau)|^{2\sigma}u(\tau) - \mu[u(\tau)] u(\tau)) d\tau \\
	& - \frac{i+\gamma}{1+\gamma^2} \int_0^t U_\Omega(t - \tau)(\nabla -\ii\Omega \wedge x) \mathcal H(u) (\tau) d\tau. 
\end{split}
\end{equation*}
Since $\nabla V$ is linear, the embedding $\Sigma(\R^d) \hookrightarrow L^r(\R^d) $ and H\"older's inequality imply that
\begin{equation*}
	\begin{aligned}
		\| \nabla \mathcal H (u) \|_{L^q([0,T),L^r) \cap L^\infty([0,T],L^2)} & \lesssim \| \nabla u_0\|_{L^2} + T^{\frac{q-q^{\prime}}{q q^{\prime}}} \| u\|^{2\sigma}_{L^\infty([0,T],H^1)} \| \nabla u\|_{L^q([0,T),L^r)} \\
		&+ T\| \mu[u] \|_{L^\infty[0,T]} \|\nabla u\|_{L^\infty([0,T],L^2)} \\
		& + T\| x \mathcal H(u)\|_{L^\infty([0,T],L^2)} + T\| \nabla \mathcal H(u)\|_{L^\infty([0,T],L^2)},
	\end{aligned}
\end{equation*}
and
\begin{equation*}
	\begin{split}
		\| x\mathcal H (u) \|_{L^q([0,T),L^r) \cap L^\infty([0,T],L^2)} &\lesssim \| x u_0\|_{L^2} + T^{\frac{q-q^{\prime}}{q q^{\prime}}} \| u\|^{2\sigma}_{L^\infty([0,T],H^1)} \| x u\|_{L^q([0,T),L^r)} \\
		&+ T\| \mu[u] \|_{L^\infty[0,T]} \|\nabla u\|_{L^\infty([0,T],L^2)} \\
		& + T\| x \mathcal H(u)\|_{L^\infty([0,T],L^2)} + T\| \nabla \mathcal H (u)\|_{L^\infty([0,T],L^2)}.
	\end{split}
\end{equation*}
It follows from Proposition \ref{prp:strich}, \eqref{eq:omegaLGP} and the estimates above that there exist $C_1>0$ and $K(M,N) >0$ such that
\begin{equation*}
\begin{aligned}
& \| \mathcal H (u) \|_{L^\infty([0,T],\Sigma) \cap L^{q}\left((0, T), W^{1, r}\right)} + \| x\mathcal H (u) \|_{L^q([0,T),L^r)} \\ 
& \leq C_1 \bigg( \|u_0\|_{\Sigma}+\left(T+T^{\frac{q-q^{\prime}}{q q'}}\right)K M \bigg),
\end{aligned}
\end{equation*}
and 
\begin{equation*}
	\begin{aligned}
		d(\mathcal H(u),\mathcal H(v)) \leq C_1\left(T+T^{\frac{q-q^{\prime}}{q q^{\prime}}}\right)K \mathrm{d}(u, v) .
	\end{aligned}
\end{equation*}
Note that
\begin{equation*}
	\frac{q-q^{\prime}}{q q^{\prime}}=1-\frac{2}{q}= \frac{2 + 2\sigma - d\sigma }{2\sigma+2}>0 .
\end{equation*}
We set $
M=2 C_1\|u_0\|_{\Sigma}$
and we choose $T$ small enough so that the following inequality 
\begin{equation}\label{eq:tempGP1}
	C_1\left(T+T^{\frac{q-q^{\prime}}{q q'}}\right)K \leq \frac{1}{2}
\end{equation} 
is satisfied. Finally, notice that \eqref{eq:GPlinL2}, \eqref{eq:GPlinL3}, \eqref{eq:rotenCN} and \eqref{eq:str11GP} imply that for any $t \in [0,T]$, there exist $C_2, C_3 >0$ such that 
\begin{equation*}
 \begin{aligned}
 \| \mathcal H(u)(t) \|_{L^2} &\geq \| U_\Omega(t) u_0 \|_{L^2} \\ 
 &- C_2 \bigg\| \int_0^t U_\Omega(t - \tau)(g |u(\tau)|^{2\sigma}u(\tau) - \mu[u(\tau)] u(\tau)) \, d\tau \bigg\|_{L^{\infty}([0,T],L^2)} \\
 & \geq \bigg( \| u_0 \|_{L^2}^2 - \frac{2 \gamma}{1 + \gamma^2} \int_0^t \big( H_\Omega (U_\Omega(\tau) u_0), U_\Omega(\tau) u_0 \big) \, d\tau \bigg)^\frac{1}{2} \\
 & - C_2 \big( T^{\frac{q-q^{\prime}}{q q^{\prime}}} \| u\|^{2\sigma}_{L^\infty([0,T],\Sigma)} \| u\|_{L^q([0,T),L^r)}  \\
 & \hspace{1cm}+ T \| \mu[u] \|_{L^\infty[0,T]} \| u\|_{L^\infty([0,T],L^2)} \big) \\
 & \geq \big( \| u_0 \|_{L^2}^2 - C_3 T \| u_0\|_{L^\infty([0,T],\Sigma)}^2 \big)^\frac{1}{2} \\
 & - C_2 \big( T^{\frac{q-q^{\prime}}{q q^{\prime}}} \| u\|^{2\sigma}_{L^\infty([0,T],\Sigma)} \| u\|_{L^q([0,T),L^r)} \\
 & \hspace{1cm}+ T \| \mu[u] \|_{L^\infty[0,T]} \| u\|_{L^\infty([0,T],L^2)} \big).
 \end{aligned}
\end{equation*}
Exploiting \eqref{eq:omegaLGP} in the inequality above implies that there exists $K_1 (N,M)>0$ such that 
\begin{equation*}
 \begin{aligned}
 \inf_{t \in [0,T]} \| \mathcal H(u) \|_{L^2} \geq \| u_0 \|_{L^2} - K_1 (T^\frac{1}{2} + T^{\frac{q-q^{\prime}}{q q^{\prime}}} + T).
 \end{aligned}
\end{equation*}
Now we set $N = \| u_0 \|_{L^2}$. By choosing $T >0$ which satisfies condition \eqref{eq:tempGP1} and also 
\begin{equation*}
 (T^\frac{1}{2} + T^{\frac{q-q^{\prime}}{q q^{\prime}}} + T) \leq \frac{N}{2 K_1}
\end{equation*}
we obtain that $\mathcal H$ maps $\mathcal E$ into itself and it is a contraction.
Thus $\mathcal H$ admits a fixed point $\psi \in C([0,T],\Sigma(\R^d))$, which is a solution to the Cauchy problem \eqref{eq:gpe}. The uniqueness follows from the fact that the map $\mathcal H$ is a contraction in $\mathcal E$. Moreover, notice that we can extend the local solution until the $\Sigma$-norm of $\psi(t)$ is bounded. Hence we obtain the blow-up alternative, that is, if $T_{max}>0$ is the maximal time of existence, then $T_{max}= \infty$, or $T_{max}< \infty$ and
	\begin{equation*}
		\lim_{t \rightarrow T_{max}} \| \psi(t)\|_{\Sigma} = \infty.
	\end{equation*}
\end{proof}

\begin{remark}
Let us emphasize why we need the condition $\sigma \geq \frac{1}{2}$. On one hand, in the absence of the nonlocal term $\mu$, we could use the contraction principle in the set $\mathcal E$ with the weaker distance 
\begin{equation*}
 \mathrm{g}(u, v)=\|u-v\|_{L^{q}\left((0, T), L^r\right)}+\|u-v\|_{L^{\infty}\left((0, T), L^2\right)}.
\end{equation*}
Indeed, it is standard to prove that $(\mathcal E,\mathrm{g})$ is a complete metric space. On the other hand, the presence of $\mu$ requires a stronger distance induced by the $L^\infty_t\Sigma$-norm, as it is clear from inequality \eqref{eq:omega1GP}. This implies that we have to use the distance $\mathrm{d}$ instead of $\mathrm{g}$. But for $\sigma < \frac{1}{2}$, this is not possible because the power-type nonlinearity is not locally Lipschitz continuous in Sobolev spaces (inequality \eqref{eq:omega2GP} fails).
\end{remark}

\subsection{Local well-posedness for $\sigma<\frac12$} \label{sec:secondLocal}

In this subsection, we address the local well-posedness problem in the case $0<\sigma<\frac12$. Here we adopt a different strategy, based on an iterative argument that constructs a sequence of approximating solutions. We start by setting $\psi^{(0)}\equiv0$. For $k\geq1$, let us assume that we already constructed $\psi^{(k-1)}\in C([0,T),\Sigma(\R^d))$. 
We then define the approximating solution $\psi^{(k)}$ as the solution to the following problem
\begin{equation*}
\begin{cases}
 &(\ii - \gamma)\partial_t \psi^{(k)} = -\frac{1}{2}\Delta \psi^{(k)} +V\psi^{(k)} + g|\psi^{(k)}|^{2\sigma} \psi^{(k)} - \Omega \cdot L\psi^{(k)} - \mu[\psi^{(k-1)}] \psi^{(k)}, \\
 &\psi^{(k)}(0) = \psi_0 \in \Sigma(\R^d). 
\end{cases} 
\end{equation*}
Notice that in the equation above, the term $\mu[\psi^{(k-1)}]$ is a $L^\infty$ function of time not depending on $\psi^{(k)}$. Thus we can use the classical contraction principle to prove the local existence of $\psi^{(k)}$ for the whole range $\sigma \in \left[0,\frac{2}{(d-2)^+}\right)$. We will then prove the convergence of the sequence $\{ \psi^{(k)}\}$ to a profile $\psi$. In order to show that $\psi$ is indeed a solution to \eqref{eq:gpe}, we will use the compact embedding $\Sigma^2(\R^d) \hookrightarrow \Sigma(\R^d)$ and the supposition $|\Omega| < \frac{\omega}{\sqrt{2}}$. This supposition is needed to show that $\psi \in L^2([0,T),\Sigma(\R^d))$, and, in particular, that $\mu [\psi(t)]$ is well defined, see Proposition \ref{prp:lwp}. Last, we will prove the uniqueness of solutions using a classical energy estimate. \newline 
In view of the description above, we start by studying the simplified model 
\begin{equation}\label{eq:simp_gp}
\begin{cases}
 (\ii - \gamma)\partial_t \psi = -\frac{1}{2}\Delta \psi +V\psi + g|\psi|^{2\sigma} \psi - \Omega \cdot L\psi - \lambda\psi,\\
 \psi(0) = \psi_0 \in \Sigma(\R^d)
\end{cases}
\end{equation}
where $\lambda$ is a $L^\infty$-function of time. Using the Strichartz estimates in Proposition \ref{prp:strich}, we will prove the local well-posedness of equation \eqref{eq:simp_gp}.

\begin{prop}\label{prp:contr}
	Let $0\leq \sigma < \frac{2}{(d-2)^+}$ and $\lambda \in L^\infty(\R)$. Then for any $\psi_0\in \Sigma(\R^d)$, there exists $0 < T = T(\|\psi_0\|_{\Sigma}, \| \lambda \|_{L^\infty})$, and a unique solution $\psi \in C([0,T), \Sigma(\R^d))$ to \eqref{eq:simp_gp}.
\end{prop}

\begin{proof}
For any $u_0 \in \Sigma$, we define the map $\mathcal F(u_0)(u)(t) = \mathcal F(u)(t)$ as 
	\begin{equation*}
		\mathcal F(u)(t) = U_\Omega(t) u_0 -\frac{i+\gamma}{1+\gamma^2} \int_0^t U_\Omega(t - \tau)(g |u(\tau)|^{2\sigma}u(\tau) - \lambda(\tau) u(\tau)) \, d\tau.
	\end{equation*}
We fix $T >0$ which will be chosen later, and let $(r,q)$ be defined as in \eqref{eq:rQThetaGp}.
We consider the set
\begin{equation*}
	X_T = \left\{ u \in C([0,T],\Sigma(\R^d)); \, u,xu,\nabla u \in L^{q}\left( (0,T), L^{r}(\R^d)\right) \right\},
\end{equation*} 
endowed with the distance 
\begin{equation*}
 d(u,v) = \| u - v\|_{L^{\infty}([0,T],L^2)} + \| u - v\|_{L^q((0,T), L^r)}.
\end{equation*}
It is a standard procedure to prove that the set $(X_T,d)$ is a complete metric space. Let $u\in X_T$.
Proposition \ref{prp:strich}, the embedding $\Sigma(\R^d) \hookrightarrow L^r(\R^d)$ and H\"older's inequality imply that
\begin{equation}\label{eq:str11}
	\begin{split}
	\| \mathcal F(u)\|_{L^q([0,T),L^r) \cap L^\infty([0,T],L^2)} &\lesssim \| u_0\|_{L^2} + T^{\frac{q-q^{\prime}}{q q^{\prime}}} \| u\|^{2\sigma}_{L^\infty([0,T],H^1)} \| u\|_{L^q([0,T),L^r)} \\ 
	&+ T \| \lambda\|_{L^\infty[0,T]} \| u\|_{L^\infty([0,T],L^2)}.
	\end{split}
\end{equation}
Moreover, using the commutator estimates in Proposition \ref{prp:comm}, we obtain
\begin{equation*}
	\begin{split}
		\nabla \mathcal F(u) &= U_\Omega(t) \nabla u_0 -\frac{i+\gamma}{1+\gamma^2} \int_0^t U_\Omega(t - \tau)\nabla(g |u(\tau)|^{2\sigma}u(\tau) - \lambda(\tau) u(\tau)) d\tau \\
		& - (\ii + \gamma) \int_0^t U_\Omega(t - \tau)(\nabla V -\ii\Omega \wedge \nabla) \mathcal F(u) (\tau) d\tau 
	\end{split}
\end{equation*}
and
\begin{equation*}
	\begin{split}
	x \mathcal F(u) &= U_\Omega(t) x u_0 -\frac{i+\gamma}{1+\gamma^2} \int_0^t U_\Omega(t - \tau)\nabla(g |u(\tau)|^{2\sigma}u(\tau) - \lambda(\tau) u(\tau)) d\tau \\
	& -\frac{i+\gamma}{1+\gamma^2} \int_0^t U_\Omega(t - \tau)(\nabla -\ii\Omega \wedge x) \mathcal F(u) (\tau) d\tau. 
\end{split}
\end{equation*}
Since $\nabla V$ is linear, the embedding $H^1(\R^d) \hookrightarrow L^r(\R^d)$ and H\"older's inequality imply that
\begin{equation*}
	\begin{aligned}
		\| \nabla \mathcal F (u) \|_{L^q([0,T),L^r) \cap L^\infty([0,T],L^2)} & \lesssim \| \nabla u_0\|_{L^2} + T^{\frac{q-q^{\prime}}{q q^{\prime}}} \| u\|^{2\sigma}_{L^\infty([0,T],H^1)} \| \nabla u\|_{L^q([0,T),L^r)} \\
		&+ T\| \lambda\|_{L^\infty[0,T]} \|\nabla u\|_{L^\infty([0,T],L^2)} \\
		& + T\| x \mathcal F(u)\|_{L^\infty([0,T],L^2)} + T\| \nabla \mathcal F(u)\|_{L^\infty([0,T],L^2)},
	\end{aligned}
\end{equation*}
and
\begin{equation*}
	\begin{split}
		\| x\mathcal F (u) \|_{L^q([0,T),L^r) \cap L^\infty([0,T],L^2)} &\lesssim \| x u_0\|_{L^2} + T^{\frac{q-q^{\prime}}{q q^{\prime}}} \| u\|^{2\sigma}_{L^\infty([0,T],H^1)} \| x u\|_{L^q([0,T),L^r)} \\
		&+ T\| \lambda\|_{L^\infty[0,T]} \|\nabla u\|_{L^\infty([0,T],L^2)} \\
		& + T\| x \mathcal F(u)\|_{L^\infty([0,T],L^2)} + T\| \nabla \mathcal F (u)\|_{L^\infty([0,T],L^2)}.
	\end{split}
\end{equation*}
By choosing $T$ to be sufficiently small, the function $\mathcal F$ maps a suitable ball in
$X_T$ into itself. Using Proposition \ref{prp:strich} as in the derivation of \eqref{eq:str11} we can also prove that $\mathcal F$ is a contraction with respect to the distance $d$. Thus, by the fixed point theorem, we conclude that for any $\psi_0 \in \Sigma$, there exists a unique solution to \eqref{eq:simp_gp}. 
\end{proof}

We will also need the following lemma in order to find a suitable bound on the approximating sequence.

\begin{lem}
 For any $|\Omega| < \frac{\omega}{\sqrt{2}}$ and $\psi \in \Sigma(\R^d)$, there exists a constant $C(\Omega,\omega)>0$ such that 
 \begin{equation}\label{eq:HSigma}
 \| \psi\|_{\Sigma^2}^2 \leq C \left((H_\Omega \psi, H_\Omega \psi) + \| \psi\|_{\Sigma}^2\right).
 \end{equation}
\end{lem}

\begin{proof}
 By recalling the definition of the operator $H_\Omega$ \eqref{eq:lin_opGP}, we compute
 \begin{equation}\label{eq:scPrH}
 \begin{aligned}
 (H_\Omega \psi, H_\Omega \psi) &= \frac{1}{4} \| \Delta \psi\|_{L^2}^2 + \frac{\omega^4}{4} \left\| |x|^2\psi\right\|_{L^2}^2 + \left\| \Omega \cdot L\psi \right\|_{L^2}^2 - \frac{\omega^2}{2} \left( \Delta \psi, |x|^2 \psi \right) \\
 & + \left(\Delta \psi, \Omega \cdot L \psi
 \right) -\omega^2 \left( |x|^2 \psi, \Omega \cdot L \psi \right) \\
 &= \frac{1}{4} \| \Delta \psi\|_{L^2}^2 + \frac{\omega^4}{4} \left\| |x|^2\psi\right\|_{L^2}^2 + \left\| \Omega \cdot L\psi \right\|_{L^2}^2 + \frac{\omega^2}{2} \big\| |x| |\nabla \psi| \big\|_{L^2}^2 \\
 & + \omega^2 (\nabla \psi, x \psi) + \left(\Delta \psi, \Omega \cdot L \psi
 \right) -\omega^2 \left( |x|^2 \psi, \Omega \cdot L \psi \right),
 \end{aligned}
 \end{equation}
 where we suppose that $\omega_j = \omega$ for any $j$, the other case being similar. The fifth term on the right-hand side of \eqref{eq:scPrH} is bounded by 
 \begin{equation}\label{eq:quaEst0}
 \omega^2 (\nabla \psi, x\psi) \leq C\| \psi\|_{\Sigma}^2.
 \end{equation}
 For the last two terms in \eqref{eq:scPrH}, we use Young's inequality to obtain 
 \begin{equation}\label{eq:eps1}
 \left( \Delta \psi, \Omega \cdot L \psi \right) \leq \| \Delta \psi\|_{L^2}\| \Omega \cdot L\psi \|_{L^2} \leq \frac{\varepsilon}{2}\|\Delta \psi \|_{L^2}^2 + \frac{1}{2\varepsilon} \| \Omega \cdot L\psi \|_{L^2}^2
 \end{equation}
 and 
 \begin{equation}\label{eq:eps2}
 \omega^2 \left(|x|^2 \psi, \Omega \cdot L\psi \right) \leq \omega^2 \big\| |x|^2 \psi \big\|_{L^2} \big\| \Omega \cdot L\psi \big\|_{L^2} \leq \frac{\varepsilon_1 \omega^4}{2}\||x|^2 \psi \|_{L^2}^2 + \frac{1}{2\varepsilon_1} \| \Omega \cdot L\psi \|_{L^2}^2.
 \end{equation}
 for some $\varepsilon, \varepsilon_1 >0$ which will be chosen later. Equation \eqref{eq:scPrH} and estimates \eqref{eq:eps1}, \eqref{eq:eps2} imply that 
 \begin{equation}\label{eq:quadest}
 \begin{aligned}
 A(\psi,\psi)&:= \big( H_\Omega \psi, H_\Omega \psi) - \omega^2 (\nabla \psi, x\psi) \\ 
 &\geq \frac{1}{2} \left(\frac{1}{2} - \eps\right) \big\| \Delta \psi \big\|_{L^2}^2 + \frac{\omega^4}{2} \left(\frac{1}{2} - \eps_1 \right) \big\| |x|^2 \psi \big\|_{L^2}^2 \\
 &+ \frac{\omega^2}{2} \big\| |x| |\nabla \psi| \big\|_{L^2}^2 + \left(1 - \frac{1}{2 \eps} - \frac{1}{2 \eps_1} \right) \big\| \Omega \cdot L \psi \big\|_{L^2}^2. 
 \end{aligned}
 \end{equation}
 In particular, if we fix $\eps = \frac{1}{2}$ in the inequality above, we get
 \begin{equation}\label{eq:quaEst12}
 \begin{aligned}
 A(\psi,\psi) \geq \frac{\omega^4}{2} \left(\frac{1}{2} - \eps_1 \right) \big\| |x|^2 \psi \big\|_{L^2}^2 + \frac{\omega^2}{2} \big\| |x| |\nabla \psi| \big\|_{L^2}^2 - \frac{1}{2 \eps_1} \big\| \Omega \cdot L \psi \big\|_{L^2}^2.
 \end{aligned}
 \end{equation}
 By choosing $\eps_1 = \frac{1}{2}$, we consequently obtain that
 \begin{equation}\label{eq:quaEst1}
 A(\psi,\psi) \geq \frac{\omega^2}{2} \big\| |x| |\nabla \psi| \big\|_{L^2}^2 - \big\| \Omega \cdot L \psi \big\|_{L^2}^2.
 \end{equation}
 From the definition of $L$ in \eqref{eq:ang_mom}, we observe that
 \begin{equation}\label{eq:estRot}
 \big\| \Omega \cdot L \psi \big\|_{L^2}^2 = \int |\Omega \cdot (\ii x \wedge \nabla \psi)|^2\, dx \leq |\Omega|^2 \big\| |x| |\nabla \psi| \big\|_{L^2}^2.
 \end{equation}
 Thus, from \eqref{eq:quaEst1}, we get
 \begin{equation}\label{eq:sigma21}
 A(\psi,\psi) \geq \left( \frac{\omega^2}{2} - |\Omega|^2 \right) \big\| |x| |\nabla \psi| \big\|_{L^2}^2 \geq C(\omega, \Omega) \big\| |x| |\nabla \psi| \big\|_{L^2}^2, 
 \end{equation}
 where $C(\omega, \Omega) >0$ due to the supposition that $|\Omega| < \frac{\omega}{\sqrt{2}}$. \newline
 Moreover, if we choose $\eps_1 = \frac{|\Omega|^2}{\omega^2}$ in inequality \eqref{eq:quaEst12}, then we obtain that
 \begin{equation*}
 A(\psi,\psi) \geq \frac{\omega^4}{2} \left(\frac{1}{2} - \frac{|\Omega|^2}{\omega^2} \right) \big\| |x|^2 \psi \big\|_{L^2}^2 + \frac{\omega^2}{2} \big\| |x| |\nabla \psi| \big\|_{L^2}^2 - \frac{\omega^2}{2 |\Omega|^2} \big\| \Omega \cdot L \psi \big\|_{L^2}^2.
 \end{equation*}
 From this estimate and by exploiting \eqref{eq:estRot}, it follows that
 \begin{equation}\label{eq:sigma22}
 A(\psi,\psi) \geq \frac{\omega^4}{2} \left(\frac{1}{2} - \frac{|\Omega|^2}{\omega^2} \right) \big\| |x|^2 \psi \big\|_{L^2}^2 \geq C_1(\omega,\Omega) \big\| |x|^2 \psi \big\|_{L^2}^2,
 \end{equation}
where we again have that $C_1(\omega,\Omega) >0$ because $|\Omega| < \frac{\omega}{\sqrt{2}}$. \newline
Furthermore, if in inequality \eqref{eq:quadest}, we choose $\eps_1 = \frac{1}{2}$, then we obtain that
\begin{equation*}
 A(\psi,\psi) \geq \frac{1}{2}\left( \frac{1}2 - \eps\right) \| \Delta \psi\|_{L^2}^2 + \frac{\omega^2}{2} \big\| |x| |\nabla \psi| \big\|_{L^2}^2 - \frac{1}{2 \eps} \| \Omega \cdot L \psi \|_{L^2}^2. 
 \end{equation*}
In this estimate we choose $\eps = \frac{|\Omega|^2}{\omega^2}$, so that
\begin{equation*}
 A(\psi,\psi) \geq \frac{1}{2}\left( \frac{1}2 - \frac{|\Omega|^2}{\omega^2}\right) \| \Delta \psi\|_{L^2}^2 + \frac{\omega^2}{2} \big\| |x| |\nabla \psi| \big\|_{L^2}^2 - \frac{\omega^2}{2 |\Omega|^2} \| \Omega \cdot L \psi \|_{L^2}^2. 
 \end{equation*}
 Exploiting \eqref{eq:estRot}, we obtain that 
 \begin{equation}\label{eq:sigma23}
 A(\psi,\psi) \geq \frac{1}{2}\left( \frac{1}2 - \frac{|\Omega|^2}{\omega^2}\right) \| \Delta \psi\|_{L^2}^2 \geq C_1(\omega, \Omega) \| \Delta \psi\|_{L^2}^2.
 \end{equation}
 Gathering together the three estimates \eqref{eq:sigma21}, \eqref{eq:sigma22}, \eqref{eq:sigma23}, we see that there exists a constant $C_2(\omega, \Omega) >0$ such that
 \begin{equation*}
 \| \Delta \psi\|_{L^2}^2 + \big\| |x|^2 \psi \big\|_{L^2}^2 + \big\| |x| |\nabla \psi| \big\|_{L^2}^2 \leq C_2 A(\psi,\psi).
 \end{equation*}
 Using \eqref{eq:normSig2} and \eqref{eq:quaEst0}, we conclude that
 \begin{equation*}
 \| \psi\|_{\Sigma^2}^2 \leq C_2\left(\big(H_\Omega \psi, H_\Omega \psi\big) + C\| \psi \|_{\Sigma}^2 \right).
 \end{equation*}
\end{proof}

As a direct consequence of this lemma, we find the following a priori estimates on a solution to \eqref{eq:simp_gp}:
\begin{prop}\label{prp:t*}
	 Let $\psi_0 \in \Sigma(\R^d)$ and $\psi \in C([0,T], \Sigma(\R^d))$ be the corresponding solution to \eqref{eq:simp_gp}. Then there exists $0 < T^*(\|\psi_0\|_{H^1},\| \lambda \|_{L^\infty}) \leq T$, such that
	\begin{equation}\label{eq:cont11}
		 \| \psi\|_{L^\infty([0,T^*],\Sigma)} \leq 2 \| \psi_0\|_{\Sigma},
	\end{equation} 
	and
	\begin{equation}\label{eq:cont12}
		 \inf_{[0,T^*]}\| \psi\|_{L^2} \geq \frac{1}{2} \| \psi_0\|_{L^2}.
	\end{equation} 
	Moreover, there exists a constant $0<\mathcal K(\| \psi_0\|_\Sigma )$ such that
	\begin{equation}\label{eq:cont13}
		 \left\|\mu[\psi]\right\|_{L^\infty[0,T^*]} \leq \mathcal K.
	\end{equation} 
	Finally, we also have 
	\begin{equation}\label{eq:cont14}
	 \partial_t \psi \in L^2\left([0,T^*], L^2(\R^d) \right), 
	\end{equation}
 and if $|\Omega| < \frac{\omega}{\sqrt{2}} $, then 
 \begin{equation}\label{eq:cont15}
	 \psi \in L^2([0,T^*],\Sigma^2(\R^d)).
	\end{equation}
\end{prop} 

\begin{proof}
 For any $\psi_0 \in \Sigma(\R^d)$, Proposition \ref{prp:contr} implies that there exists a time $T>0$ and a solution $\psi \in C([0,T), \Sigma(\R^d))$ to \eqref{eq:simp_gp}. In particular, by continuity, we can find a time $0 < T^* = T^*(\|\psi_0\|_{H^1},\| \lambda \|_{L^\infty}) \leq T$ such that \eqref{eq:cont11} and \eqref{eq:cont12} are true.\newline 
Then, by using \eqref{eq:rotenCN}, \eqref{eq:cont11},\eqref{eq:cont12} and the Gagliardo-Nirenberg inequality, we obtain that 
 \begin{equation*}
 \begin{aligned}
 \left\|\mu[\psi]\right\|_{L^\infty[0,T^*]} \lesssim \sup_{t \in [0,T^*]} \left( \| \psi(t)\|_\Sigma^2 + \| \psi(t)\|_{L^{2\sigma +2}}^{2\sigma +2} \right) \leq C
 \end{aligned}
 \end{equation*}
where $C(\|\psi_0\|_{\Sigma}) >0$. \newline
 To prove \eqref{eq:cont14} and \eqref{eq:cont15}, let us for the moment assume sufficient regularity and spatial decay of $\psi$ so that the following computations are justified. We take the scalar product of \eqref{eq:simp_gp} with $\psi$ and obtain that
 \begin{equation}\label{eq:massL}
 \gamma \frac{d}{dt} \| \psi \|_{L^2}^2 = 2 \left(\lambda - \mu[\psi]\right) \| \psi\|_{L^2}^2.
 \end{equation}
 Similarly, by taking the scalar product of equation \eqref{eq:simp_gp} with $\partial_t \psi$, integrating in time, using \eqref{eq:cont13}, \eqref{eq:massL} and Gagliardo-Nirenberg's inequality we obtain 
 \begin{equation}\label{eq:eneL}
 \begin{aligned}
 2 \gamma \int_0^{T^*} \| \partial_\tau \psi(\tau) \|_{L^2}^2 \, d\tau &= E[\psi_0] - E[\psi(T^*)] + \int_0^{T^*} \lambda(\tau) \frac{d}{d\tau} \| \psi(\tau)\|_{L^2}^2 \, d\tau \\
 & = E[\psi_0] - E[\psi(T^*)] + \frac{2}{\gamma}\int_0^{T^*} \lambda (\lambda - \mu[\psi]) \| \psi\|_{L^2}^2 \, d\tau \\
 &\lesssim C(\|\psi_0\|_{\Sigma}) 
 \\
 &+ \| \lambda\|_{L^\infty[0,T^*]} \big( \| \lambda \|_{L^\infty[0,T^*]} + \| \mu[\psi]\|_{L^\infty[0,T^*]}\big) \\
 &\times T^* \| \psi\|_{L^\infty([0,T^*],L^2)}^2 \\
 &\leq C(\|\psi_0\|_{\Sigma}, \| \lambda \|_{L^\infty} ) < \infty.
 \end{aligned}
 \end{equation}
 This implies that $\partial_t \psi \in L^2\left([0,T^*], L^2 (\R^d)\right)$. \newline
 Finally, we observe that \eqref{eq:simp_gp} implies formally that
 \begin{equation*}
 \begin{aligned}
 (\gamma^2 +1) \int_0^{T^*}\| \partial_\tau \psi\|_{L^2}^2 \, d\tau &= \int_0^{T^*} (H_\Omega \psi, H_\Omega \psi) \, d\tau \\ &+ \int_0^{T^*}\int g^2|\psi|^{4\sigma + 2} + \lambda^2 |\psi|^2 \, dx \, d\tau \\
 &+ 2\int_0^{T^*} (H_\Omega \psi, g|\psi|^{2\sigma}\psi - \lambda \psi) \, d\tau,
 \end{aligned}
 \end{equation*}
 which is equivalent to 
 \begin{equation}\label{eq:sigma2Der}
 \int_0^{T^*} (H_\Omega \psi, H_\Omega \psi) \, d\tau \leq (\gamma^2 +1) \int_0^{T^*}\| \partial_\tau \psi\|_{L^2}^2 \, d\tau - 2\int_0^{T^*} (H_\Omega \psi, g|\psi|^{2\sigma}\psi - \lambda \psi) \, d\tau.
 \end{equation}
The last term in \eqref{eq:sigma2Der} is estimated as
\begin{equation*}
 \begin{aligned}
 \left| 2\int_0^{T^*} \left(H_\Omega \psi, \lambda \psi\right) d\tau\right| \lesssim \| \lambda\|_{L^\infty[0,T^*]} \| \psi\|_{L^\infty([0,T^*],\Sigma)}^2.
 \end{aligned}
\end{equation*}
Now we suppose that $d = 3$, (the case $d= 2$ is similar). Then, by using H\"older's inequality, we estimate the third term in \eqref{eq:sigma2Der} as
\begin{equation*}
 \begin{aligned}
 \left| \left(H_\Omega \psi, g|\psi|^{2\sigma}\psi \right) \right| &\lesssim \int |\psi|^{2\sigma} ( |x\psi|^2 + |\nabla \psi|^2)\, dx \\
 &\lesssim \| \psi \|_{L^\frac{2d}{d-2}}^{2\sigma} ( \| \nabla \psi \|_{L^r}^2 + \| x \psi \|_{L^r}^2),
 \end{aligned}
	\end{equation*}
	where 
	\begin{equation*}
		r = \frac{2d}{d-\sigma(d-2)}.
	\end{equation*}
	By using again H\"older's inequality in time and Sobolev embedding, we obtain
	\begin{equation*}
 \begin{aligned}
 \int_0^{T^*} \| \psi \|_{L^\frac{2d}{d-2}}^{2\sigma}( \| \nabla \psi \|_{L^r}^2 + \| x \psi \|_{L^r}^2) \, d\tau & \leq \| \psi \|_{L^\infty([0,T^*],H^1)}^{2\sigma} (T^*)^{1/p} \\ 
 & \times \big( \| \nabla \psi\|_{L^q([0,T^*],L^r)}^2 + \| x \psi\|_{L^q([0,T^*],L^r)}^2\big)
 \end{aligned}
	\end{equation*}
 where 
 \begin{equation*}
 q = \frac{4}{\sigma(d-2)}
 \end{equation*}
 is such that $(q,r)$ satisfies condition \eqref{eq:qrAdm} and 
 \begin{equation*}
 0 < p = \frac{q}{q-2} < \infty.
 \end{equation*}
Thus, \eqref{eq:eneL}, \eqref{eq:sigma2Der}, \eqref{eq:cont11} and \eqref{eq:cont12} imply that 
 \begin{equation}\label{eq:stima2GP}
 \begin{aligned}
 \int_0^{T^*} \left(H_\Omega \psi,H_\Omega \psi \right) \, d\tau & \lesssim \| \partial_t \psi\|_{L^2([0,T],L^2)}^2 + \| \lambda\|_{L^\infty[0,T^*]} \| \psi\|_{L^\infty([0,T^*],\Sigma)}^2 \\
 &+ \| \psi \|_{L^\infty([0,T^*],H^1)}^{2\sigma} (T^*)^{1/p} \\
 & \times ( \| \nabla \psi\|_{L^q([0,T^*],L^r)}^2 + \| x \psi\|_{L^q([0,T^*],L^r)}^2) \leq C
 \end{aligned}
\end{equation}
where $C = C(\|\psi_0\|_{\Sigma}, \| \lambda\|_{L^\infty}) >0$.
On the other hand, if $|\Omega| < \frac{\omega}{\sqrt{2}} $, then by using \eqref{eq:HSigma} we have 
\begin{equation}\label{eq:stima1GP}
 \int_0^{T^*} \| \psi(\tau) \|_{\Sigma^2}^2 \, d\tau \leq C \left(\int_0^{T^*} \left( H_\Omega \psi,H_\Omega \psi \right) d\tau + T^* \| \psi\|_{L^\infty([0,T^*],\Sigma)}^2 \right).
\end{equation}
Thus, by combining \eqref{eq:stima1GP}, \eqref{eq:stima2GP} and \eqref{eq:cont11} we have that
\begin{equation*}
 \int_0^{T^*} \| \psi(\tau) \|_{\Sigma^2}^2 \, d\tau \leq C
\end{equation*}
where $C$ is a constant depending on $\|\psi_0\|_{\Sigma},\omega$ and $\Omega$. 
\end{proof}

Propositions \ref{prp:contr} and \ref{prp:t*} allow us to infer the local existence result for \eqref{eq:gpe}. The uniqueness will be given in Proposition \ref{prp:uniGP} below. 

\begin{prop}\label{prp:lwp}
	Let $0 < \sigma < \frac{2}{(d-2)^+}$ and $|\Omega| < \frac{\omega}{\sqrt{2}}$. For any $\psi_0 \in \Sigma(\R^d)$, there exists a maximal time of existence $ T_{max} > 0$ and a solution $\psi \in C([0,T_{max}), \Sigma(\R^d))$ to \eqref{eq:gpe}. Moreover, either $T_{max} = \infty$, or $T_{max} < \infty$ and
	\begin{equation*}
		\lim_{t \rightarrow T_{max}} \| \nabla \psi(t)\|_{\Sigma} = \infty.
	\end{equation*}
\end{prop}

\begin{proof}
	 Let $\psi_0 \in \Sigma(\R^d), \ \psi_0 \centernot{\equiv} 0.$ Let $\psi^{(0)} = \psi_0$ and for any $k \in \N$, let $\psi^{(k+1)}$ be defined as the local solution to the following Cauchy problem:
	\begin{equation}\label{eq:gpk}
		\begin{cases}
			 (\ii - \gamma)\partial_t \psi^{(k+1)} 
 &= -\frac{1}{2}\Delta \psi^{(k+1)} +V\psi^{(k+1)} + g|\psi^{(k+1)} |^{2\sigma} \psi^{(k+1)} \\ 
 &- \Omega \cdot L\psi^{(k+1)} - \mu[\psi^{(k)}] \psi^{(k+1)} , \\
			 \psi^{(k+1)} (0) = \psi_0.&
		\end{cases}
	\end{equation}
 By using Proposition \ref{prp:contr}, there exists a local solution $\psi^{(k)} \in C\left([0,T_k], \Sigma(\R^d)\right)$, where $0 < T_k = T\left(\| \psi_0\|_\Sigma, \| \mu[\psi^{(k-1)}]\|_{L^\infty[0,T_k]}\right)$ is a time which depends on $k$. 
	We will now show by induction that there exists $\mathcal T >0$, not depending on $k$, so that for any $k \in \N$, $\psi^{(k)} \in C\left([0,\mathcal T], \Sigma(\R^d)\right)$, and we also have
	\begin{equation}\label{eq:cont111}
		 \sup_{k\in \N} \| \psi^{(k)}\|_{L^\infty([0,\mathcal T],\Sigma)} \leq 2 \| \psi_0\|_{\Sigma},
	\end{equation} 
	and
	\begin{equation}\label{eq:cont122}
		 \inf_{k \in \N} \inf_{[0,\mathcal T]}\| \psi^{(k)}\|_{L^2} \geq \frac{1}{2} \| \psi_0\|_{L^2}.
	\end{equation} 
 Clearly, properties \eqref{eq:cont111} and \eqref{eq:cont122} are true for $k = 1$. Assume that they are true until some $n\in \N$. Then Proposition \ref{prp:t*} and inequalities \eqref{eq:cont111},\eqref{eq:cont122} imply that there exists a constant $\mathcal K(\|\psi_0\|_\Sigma) >0$ , not depending on $n$, such that 
 \begin{equation*}
 \| \mu[\psi^{(n)}] \|_{L^\infty[0,\mathcal T]} \leq \mathcal K,
 \end{equation*}
	and a time $T^*(\| \psi_0\|_\Sigma, \mathcal K) >0$ , also not depending on $n$, such that the Cauchy problem \eqref{eq:gpk} admits a local solution $\psi^{(n+1)} \in C\left([0,T^*], \Sigma(\R^d)\right)$ satisfying estimates \eqref{eq:cont11}, \eqref{eq:cont12} and \eqref{eq:cont13}. \newline
	Thus we choose $\mathcal T = T^*(\| \psi_0\|_\Sigma, \mathcal K) $, which is not depending on $k$. This implies that the whole sequence $\{\psi^{(k)}\} \subset C([0,\mathcal{T}], \Sigma(\R^d))$ satisfies conditions \eqref{eq:cont111} and \eqref{eq:cont122}. In particular, the sequence is uniformly bounded in $L^\infty([0,\mathcal{T}], \Sigma(\R^d))$ and, from \eqref{eq:cont14}, we also have that for all $k\in \N$
	\begin{equation*}
	 \partial_t \psi^{(k)} \in L^2\left([0,\mathcal T], L^2(\R^d) \right), \quad \psi^{(k)} \in L^2([0,\mathcal T],\Sigma^2(\R^d)),
	\end{equation*}
	and the sequence $\{ \psi^{(k)}\}$ is uniformly bounded in these spaces. So there exists a subsequence, still denoted by $\{\psi^{(k)}\}$ and $\psi \in L^{\infty}([0,\mathcal T], \Sigma(\R^d)) \cap L^{2}([0,\mathcal T],\Sigma^2(\R^d))$ with $\partial_t \psi \in L^{2}([0,\mathcal T], L^2(\R^d))$, such that
	\begin{equation*}
		\psi^{(k)} \overset{\ast}{\rightharpoonup} \psi \ \mbox{ in } \ L^{\infty}([0,\mathcal T], \Sigma(\R^d)) \quad \mbox{ and } \quad 	\psi^{(k)} \rightharpoonup \psi \ \mbox{ in } \ L^{2}([0,\mathcal T],\Sigma^2(\R^d)).
	\end{equation*}
 Since the embedding $\Sigma^2(\R^d) \hookrightarrow \Sigma(\R^d)$ is compact (Proposition \ref{prp:sigemb}), the two convergences imply that 
 \begin{equation*}
 \psi^{(k)} \rightarrow \psi \quad \mbox{in} \quad L^2([0,\mathcal T], \Sigma(\R^d))
 \end{equation*}
 strongly, and also that $\psi \in C([0,\mathcal T], L^2(\R^d))$.
 Moreover
	\begin{equation*}
		\mu[\psi^{(k)}(t)] \rightarrow \mu[\psi(t)] 
	\end{equation*}
	strongly in $L^2(\R)$ and $ \mbox{weakly}^*$ in $L^\infty(\R)$. Then $\psi(t)$ satisfies \eqref{eq:gpe} on $[0,\mathcal T]$. 
\newline
Notice also that the time of existence $\mathcal T$ depends only on the initial condition, so it is straightforward to obtain the blow-up alternative. Indeed, for any $T>0$, if the $\Sigma$-norm of $\psi(T)$ is finite, we could use the same process to extend the lifespan of the solution to $T + \mathcal{T}(\| \psi(T)\|_\Sigma$). Consequently, either the maximal time of existence of a solution is infinite $T_{max} = \infty$, or $T_{max} < \infty$ and 
\begin{equation}\label{eq:bu_alt}
		\lim_{t \rightarrow T_{max}} \| \psi(t)\|_{\Sigma} = \infty.
\end{equation} 
\end{proof}

We will need the following properties to prove the uniqueness of solutions.

\begin{prop}\label{prp:unipr}
 Let $0 < \sigma < \frac{2}{(d-2)^+}$, $\psi_0 \in \Sigma(\R^d)$, $|\Omega| < \frac{\omega}{\sqrt{2}}$, and let $\psi \in C([0,T_{max}), \Sigma(\R^d))$ be the corresponding solution to \eqref{eq:gpe}. Then for any $0 < T < T_{max}$ and $t \in [0,T]$, we have
 \begin{equation}\label{eq:mas}
 \| \psi(t)\|_{L^2} = \| \psi_0\|_{L^2}
 \end{equation}
	and
	\begin{equation}\label{eq:eneq}
 E[\psi(t)] = E[\psi_0] - 2\gamma \int_0^t \| \partial_\tau \psi(\tau)\|^2_{L^2} \, d\tau.
	\end{equation}
	Moreover 
	\begin{equation}\label{eq:spaces}
 \partial_t \psi \in L^2\left([0,T], L^2(\R^d) \right), \quad \psi \in L^2([0,T],\Sigma^2(\R^d)).
	\end{equation}
\end{prop}

\begin{proof}
 The proof is very similar to that given for Proposition \ref{prp:t*}. Indeed, let us assume sufficient regularity and spatial decay of $\psi$ so that the following computations are satisfied. Then, by taking the scalar product of \eqref{eq:gpe} with $\psi$, we obtain the conservation of the $L^2$-norm \eqref{eq:mas}. Similarly, by taking the scalar product of \eqref{eq:gpe} with $\partial_t \psi$, we obtain \eqref{eq:eneq}. Gagliardo-Nirenberg's inequality implies that
 \begin{equation*}
 \begin{aligned}
 2\gamma \int_0^t \| \partial_\tau \psi(\tau)\|^2_{L^2} \, d\tau &= E[\psi_0] - E[\psi(t)] \\ &\leq E[\psi_0] + C\left(\|\psi(t)\|_{\Sigma}^{2} + \|\psi(t)\|_{\Sigma}^{2\sigma + 2} \right). 
 \end{aligned} 
 \end{equation*}
 Finally, by using \eqref{eq:scPrH} and by straightforwardly adapting the last part of the proof of Proposition \ref{prp:t*}, we obtain $\psi \in L^2([0, T],\Sigma^2(\R^d))$. 
\end{proof}

Now we prove the uniqueness of solutions to equation \eqref{eq:gpe}.

\begin{prop}\label{prp:uniGP}
 Let $0 < \sigma < \frac{2}{(d-2)^+}$, $|\Omega| < \frac{\omega}{\sqrt{2}}$ and $\psi_0 \in \Sigma(\R^d)$. Then there exists a unique solution $\psi \in C([0,T_{max}), \Sigma(\R^d))$ to \eqref{eq:gpe}
\end{prop}	

\begin{proof}
 It only remains to show the uniqueness of solutions. We suppose that there exist two different solutions $\psi, \varphi \in C([0,T),\Sigma(\R^d))$ to the equation \eqref{eq:gpe}, both starting from the same initial condition $\psi_0\in \Sigma(\R^d)$. Let $T>0$ be such that 
 \begin{equation*}
 \max \left(\| \psi \|_{L^\infty([0,T], \Sigma)} , \| \varphi \|_{L^\infty([0,T], \Sigma)}\right) \leq 2\| \psi_0\|_{\Sigma}.
 \end{equation*}
 As a consequence, Proposition \ref{prp:unipr}) implies that there exists a constant $C>0$ such that
 \begin{equation*}
 \| \psi \|_{L^2([0,T],\Sigma^2)} + \| \varphi \|_{L^2([0,T],\Sigma^2)} \leq C.
 \end{equation*}
 Let $(r,q)$ be defined as in \eqref{eq:rQThetaGp}. Then from Proposition \ref{prp:strich} we get
 \begin{equation}\label{eq:str1}
 	\begin{split}
	 	&\left\| \int_0^t U_\Omega(t - \tau) \left(|\psi(\tau)|^{2\sigma}\psi(\tau) - |\varphi(\tau)|^{2\sigma} \varphi(\tau)\right) d\tau \right\|_{L^q([0,T),L^r) \cap L^\infty([0,T],L^2) } \\& \lesssim \left(\|\psi\|_{L^\infty([0,T],L^r)}^{2\sigma +1} + \| \varphi \|_{L^\infty([0,T],L^r)}^{2\sigma +1} \right) \| \psi - \varphi\|_{L^{q'}([0,T),L^r)}.
 	\end{split}
 \end{equation}
 We can formally write that
 \begin{equation*}
	 |\mu[\psi] -\mu[\varphi] | \lesssim \left( \| (-\Delta + V(x)) \psi \|_{L^2} + \| (-\Delta + V(x)) \varphi \|_{L^2} \right) \| \psi - \varphi\|_{L^2}.
 \end{equation*}
 This implies that
 \begin{equation}\label{eq:str2}
 	\begin{split}
	 	&\bigg\| \int_0^t U_\Omega(t - \tau) \left(\mu[\psi(\tau)] \psi(\tau) -\mu[\varphi(\tau)]\varphi(\tau) \right) d\tau \bigg\|_{L^q([0,T),L^r) \cap L^\infty([0,T],L^2)} \\
	 	& \lesssim 
	 	\big\| |\mu[\psi] - \mu[\varphi]| \psi \big\|_{L^1([0,T),L^2)} + \| \mu[\varphi]\|_{L^2[0,T)} \| \psi -\varphi\|_{L^2([0,T),L^2)} \\
	 	& \lesssim \big(\| \psi \|_{L^2([0,T),\Sigma^2)} + \| \varphi \|_{L^2([0,T),\Sigma^2)}\big) \| \psi -\varphi\|_{L^2([0,T),L^2)}. 
 	\end{split}
 \end{equation}
 Using the integral formulation of equation \eqref{eq:gpe}
 \begin{equation*}
 \psi(t) = U_\Omega(t) \psi_0 -\frac{i+\gamma}{1+\gamma^2} \int_0^t U_\Omega(t - \tau) \left(g |\psi(\tau)|^{2\sigma}\psi(\tau) - \mu[\psi(\tau)] \psi(\tau)\right) d\tau,
 \end{equation*}
 we obtain
 \begin{equation*}
 \begin{aligned}
 \| \psi - \varphi\|_{L^\infty([0,T],L^2)} & + \| \psi - \varphi\|_{L^q([0,T),L^r)} \\ 
 & \lesssim \| \psi - \varphi\|_{L^2([0,T),L^2)} + \| \psi - \varphi\|_{L^{q'}([0,T),L^r)} \\
 & \lesssim T^\frac{1}{2}\| \psi - \varphi\|_{L^\infty([0,T],L^2)} + T^{\frac{ q - q' }{ q q' }} \| \psi - \varphi\|_{L^q([0,T),L^r)}.
 \end{aligned}
 \end{equation*}
 Since $\frac{ q - q' }{ q q' } >0$, by choosing $T >0$ sufficiently small, it follows that 
 \begin{equation*}
 \| \psi - \varphi\|_{L^\infty([0,T],L^2)} + \| \psi - \varphi\|_{L^q([0,T),L^r)} \leq 0,
 \end{equation*}
 that is $\psi = \varphi$ in the interval $[0,T]$.
\end{proof}

We notice that the proof of Theorem \ref{thm:lwp1} is given by combining Propositions \ref{prp:lwp} and \ref{prp:uniGP}.

\subsection{Global Well-posedness}

Under the conditions for local well-posedness stated in Theorems \ref{thm:lwp2} and \ref{thm:lwp1}, with the additional supposition that $\sigma < \frac{2}{d}$ if $g <0$, we obtain the global well-posedness of solutions by a classical argument which follows from the fact that the energy is decreasing, see \eqref{eq:en_decrease}. We shall first prove the following estimate.
\begin{lem}
 If $|\Omega| < \omega$, then there exists a constant $C(\Omega,\omega) >0$ such that for any $u \in \Sigma(\R^d)$,
 \begin{equation}\label{eq:sigmaEn}
 \| u \|_{\Sigma}^2 \leq C(H_\Omega u, u).
 \end{equation}
\end{lem}

\begin{proof}
 By recalling the definition of $H_\Omega$ \eqref{eq:lin_opGP}, we observe that
 \begin{equation}\label{eq:stimaH}
 (H_\Omega u,u) = \frac12 \| \nabla u\|_{L^2}^2 + \frac{\omega^2}{2} \| xu\|_{L^2}^2 - \Omega (Lu,u),
 \end{equation}
 were we suppose that $\omega_j = \omega $ for any $j$ in \eqref{eq:V}, the other case being similar. We use Cauchy-Schwartz and Young's inequalities to obtain 
 \begin{equation*}
 |\Omega (Lu,u)| \leq |\Omega| \| \nabla u\|_{L^2} \| xu\|_{L^2} \leq \frac{\varepsilon}{2} \| \nabla u \|_{L^2}^2 + \frac{|\Omega|^2}{2\varepsilon} \| xu\|_{L^2}^2,
 \end{equation*}
 for any $\varepsilon >0$. In particular, by choosing $\varepsilon = 1$, we obtain from equation \eqref{eq:stimaH} 
 \begin{equation*}
 \frac{\omega^2 - |\Omega|^2}{2} \| x u \|_{L^2}^2 \leq (H_\Omega u,u),
 \end{equation*}
 while, by choosing $\varepsilon = \frac{|\Omega|^2}{\omega^2}$, we get 
 \begin{equation*}
 \frac{1}{2}\left( 1 - \frac{|\Omega|^2}{\omega^2}\right) \| \nabla u \|_{L^2}^2 \leq (H_\Omega u,u). 
 \end{equation*}
 Thus it follows that 
 \begin{equation*}
 \| u \|_{\Sigma}^2 \leq \frac{2}{C} (H_\Omega u,u). 
 \end{equation*}
 with
 \begin{equation*}
 C = \min \left( \frac{\omega^2 - |\Omega|^2}{2} ,\frac{1}{2}\left( 1 - \frac{|\Omega|^2}{\omega^2}\right)\right) 
 \end{equation*}
 and $C > 0$ for any $|\Omega| < \omega$. 
\end{proof}

With the lemma above and the fact the energy is decreasing \eqref{eq:en_decrease}, we can now prove we can now give a proof of Theorem \ref{thm:gwp} on the existence of global solutions. 

\begin{proof}[Proof of Theorem \ref{thm:gwp}:]
	First, we consider the case $g \geq 0$, and $\sigma < \frac{2}{(d-2)^+}$. Then estimate \eqref{eq:sigmaEn} and the definition of the energy \eqref{eq:energy_gpe} imply that
	\begin{equation*}
	 \| \psi(t)\|_\Sigma^2 \leq C(H_\Omega \psi, \psi) \leq C\big((H_\Omega \psi, \psi) + \frac{g}{\sigma + 1} \| \psi\|_{L^{2\sigma + 2}}^{2\sigma + 2}\big) = C E[\psi].
	\end{equation*}
 Since the energy is decreasing in time \eqref{eq:en_decrease}, we obtain the uniform bound 
\begin{equation}\label{eq:enbelow}
	 \| \psi(t)\|_\Sigma^2 \leq C E[\psi_0].
\end{equation}
In the same way, when $g < 0$, we can use \eqref{eq:sigmaEn}, \eqref{eq:en_decrease}, the conservation of the mass \eqref{eq:mas} and the Gagliardo-Nirenberg inequality to obtain 
 \begin{equation}\label{eq:unifg}
	 \| \psi(t)\|_\Sigma^2 \lesssim E[\psi_0] - \frac{g}{\sigma + 1} \| \psi\|_{L^{2\sigma + 2}}^{2\sigma + 2} \lesssim E[\psi_0] + \| \psi_0 \|_{L^2}^{2 + (2 -d)\sigma} \| \psi(t) \|_{\Sigma}^{d\sigma}.
	\end{equation}
 Since $d\sigma < 2$, we obtain again a uniform-in-time bound on $\| \psi(t)\|_\Sigma$. 
\end{proof}

\begin{remark}\label{rem:norms}
 Notice that \eqref{eq:sigmaEn} implies that for $u \in \Sigma(\R^d)$, $\sqrt{(H_\Omega u, u)}$ is an equivalent norm to $\| u\|_{\Sigma}$. Indeed we get
 \begin{equation}\label{eq:norms}
 \| u\|_{\Sigma}^2 \leq C (H_\Omega u, u) \leq K \| u\|_{\Sigma}^2,
 \end{equation}
 where $K = K(\omega) >0$. 
\end{remark}

\section{Asymptotic Behavior of the Linear Equation}\label{sec:linASGP}

In this section, we study the asymptotic behavior of solutions in the linear case $g= 0$. We will show that the asymptotic state depends on the initial condition. We introduce some notations to state our result. We denote the spectrum of the linear operator $H_\Omega$ by $\sigma(H_\Omega) = \{\lambda_{n}\}_{n\in \N}$ which is discrete, see Proposition \ref{prp:eigenv}. 
We order the eigenvalues in an increasing order
\begin{equation*}
	\lambda_{n} < \lambda_{m} \ \mbox{ if } \ n < m.
\end{equation*}
Let $\mathcal W_{n}$ be the eigenspaces associated with the eigenvalues $\lambda_{n}$, and $\phi_{n,k} \in \mathcal W_{n}$ the relative basis orthonormal basis of eigenfunctions, where $k=1,...,m_n$, and $m_n = \dim (\mathcal W_{n})$. Since these eigenfunctions form a complete orthonormal basis of $L^2$, we decompose the initial condition as 
\begin{equation*}
 \psi_0 = \sum_{n=1}^\infty\sum_{k=1}^{m_n} (\psi_0, \phi_{n,k}) \phi_{n,k}.
\end{equation*}
Let us denote by $\lambda_M$ the smallest eigenvalue in the decomposition of $\psi_0$ that is
\begin{equation}\label{eq:minEi}
 \lambda_{M} = \min \left\{ \lambda_{n} \in \sigma(H_\Omega) :\, \exists \phi_{n,k} \in \mathcal W_n, \, (\psi_0, \phi_{n,k}) \neq 0 \right\}. 
\end{equation}
Then we will prove that $\psi(t)$ asymptotically converges to the eigenspace $\mathcal W_M$ and $\mu[\psi(t)]$ to the eigenvalue $\lambda_{M}$. 

\begin{theorem}\label{thm:linAs}
	 Let $g = 0$, $\psi_0 \in \Sigma(\R^d)$ and let $\psi \in C([0,\infty),\Sigma(\R^d))$ be the corresponding solution to \eqref{eq:gpe}. Then we have
	\begin{equation*}
		\lim_{t \rightarrow \infty} \mu[\psi(t)] = \lambda_{M},
	\end{equation*}
	where $\lambda_{M}$ is defined in \eqref{eq:minEi}. Moreover, 
	\begin{equation*}
		\lim_{t \rightarrow \infty}\inf_{\varphi \in \mathcal W_{M}}\| \psi(t) - \varphi\|_{\Sigma} = 0. 
	\end{equation*}
\end{theorem}

\begin{proof}
 Without losing generality, we suppose that $\| \psi_0\|_{L^2} = 1$. From the conservation of the mass, this implies that for any $t>0$, $\| \psi(t)\|_{L^2} = 1$. Moreover, since $g= 0$, we also have that 
 \begin{equation}\label{eq:linear_mu}
 \mu[\psi(t)] = \left( H_\Omega \psi, \psi \right) = E[\psi(t)].
 \end{equation}
 Thus $\mu[\psi(t)]$ is decreasing in time and also bounded from below by the first eigenvalue of the operator $H_\Omega$.
 This implies that there exists $\mu_\infty \in \R$ such that 
 \begin{equation*}
 \lim_{t\rightarrow \infty} \mu[\psi(t)] = \mu_\infty.
 \end{equation*}
 
 First, we show that $\mu_\infty \in \sigma (H_\Omega)$. We observe that from \eqref{eq:sigmaEn}, \eqref{eq:linear_mu} and \eqref{eq:en_decrease}, there exists a constant $C>0$ such that for any $t \geq 0$ 
 \begin{equation*}
 \| \psi(t)\|_{\Sigma}^2 \leq C \left( H_\Omega \psi, \psi \right) \leq C E[\psi_0].
 \end{equation*}
In particular, $ \| \psi(t)\|_{\Sigma}$ is uniformly bounded in time. By Proposition \ref{prp:unipr}, this also implies that $\partial_t \psi \in L^2([0,\infty),L^2(\R^d))$. Thus there exists a $\psi_\infty \in \Sigma(\R^d)$ and a sequence of times $\{ t_j\} \subset \R^+$, $t_j \to \infty$ as $j \to \infty$ such that 
\begin{equation*}
 		\begin{cases}
 			\psi_j \rightharpoonup \psi_\infty, \quad \mbox{ in } \ \Sigma(\R^d), \\
 		 	\partial_t\psi_j \rightarrow 0, \quad \mbox{ in } \ L^2 , \\
 		 	(\ii - \gamma) \partial_t \psi_j + (\mu[\psi_j] - \mu_\infty) \psi_j = H_\Omega \psi_j - \mu_\infty \psi_j \rightarrow 0 \quad \mbox{ in } \ L^2 
 		\end{cases}
	 \end{equation*}
	as $j \rightarrow \infty$, where $\{\psi_j\} = \{\psi(t_j)\}$. Therefore, the profile $\psi_\infty$ satisfies 
 \begin{equation*}
 		 H_\Omega \psi_\infty = \mu_\infty \psi_\infty
	\end{equation*}
 in a weak sense. From the compact embedding $\Sigma(\R^d) \hookrightarrow L^2(\R^d)$, we have that $\| \psi_\infty\|_{L^2} = 1$. 
 Consequently $\mu_\infty \in \sigma(H_\Omega)$ and $\mu_\infty = \mu[\psi_\infty] = (H_\Omega \psi_\infty, \psi_\infty)$. In particular, we have 
 \begin{equation*}
 (H_\Omega \psi_j, \psi_j) \to (H_\Omega \psi_\infty, \psi_\infty).
 \end{equation*}
 Since for any $u\in \Sigma(\R^d)$, $(H_\Omega u, u)$ is an equivalent norm to $\| u\|_{\Sigma}$, see Remark \ref{rem:norms}, this implies the strong convergence $\psi_j \to \psi_\infty$ in $\Sigma(\R^d)$.
 \newline
 Next, we will show that $\mu_\infty = \lambda_M$ where $M$ is the smallest eigenvalue in the decomposition of $\psi_0$, see \eqref{eq:minEi}. We decompose $\psi(t)$ as 
 \begin{equation*}
 \psi(t,x) = \sum_{n=1}^\infty \sum_{k = 1}^{m_n} b_{n,k}(t) \phi_{n,k}(x)
 \end{equation*}
 where $b_{n,k}(t) = (\psi(t), \phi_{n,k})$. We plug this decomposition into \eqref{eq:gpe} and obtain that 
 \begin{equation}\label{eq:gpDeco}
 (\ii - \gamma) \sum_{n=1}^\infty \sum_{k = 1}^{m_n} \dot{b}_{n,k} \phi_{n,k} = \sum_{n=1}^\infty \sum_{k = 1}^{m_n} \lambda_n b_{n,k} \phi_{n,k} - \mu[\psi] \sum_{n=1}^\infty \sum_{k = 1}^{m_n} b_{n,k}(t) \phi_{n,k}(x).
 \end{equation}
 By taking the scalar product of \eqref{eq:gpDeco} with $\phi_{n,k}$ we obtain that
 \begin{equation*}
	 b_{n,k}(t) = b_{n,k}(0)\exp\left( \gamma \int_0^t \mu[\psi(\tau)] - \lambda_{n} \, d\tau\right).
 \end{equation*}
 Consequently, if $b_{n,k}(0) = 0$ then $b_{n,k}(t) = 0$ for every $t \geq 0$. Now let $\lambda_M$ be defined as in \eqref{eq:minEi}. Suppose that $\mu_\infty \neq \lambda_M$. Thus there exists $m >M$ such that $ \mu_\infty = \lambda_m > \lambda_M$. Since $\mu[\psi(t)]$ is decreasing in time, this implies that there exists $\delta>0$ so that $\mu[\psi(t)] - \lambda_{M} > \delta$ for any $t \geq 0$. Let $k \leq m_M$ be such that $b_{M,k}(0) = (\psi_0, \phi_{M,k}) \neq 0$. It follows that
 \begin{equation*}
	 |b_{M,k} (t)| = |b_{M,k}(0)| \exp \left( \gamma \int_0^t \mu[\psi(\tau)] - \lambda_{M}\, d\tau \right) > |b_{M,k}(0)| e^{\gamma \delta t},
 \end{equation*} 
 which is a contradiction with respect to the conservation of the $L^2$-norm \eqref{eq:mas}. This is enough to conclude that $\mu_\infty = \lambda_{M}$.
\newline
 Finally, we prove that the solution converges strongly in $\Sigma(\R^d)$ to the eigenspace $\mathcal W_M$ relative to the eigenvalue $\lambda_{M}$. 
 Let $\delta>0$ be such that for any $ n > M$,
 \begin{equation*}
 \lambda_{n} \geq \lambda_{M} + 2\delta. 
 \end{equation*}
 Since $\mu[\psi(t)] \to \lambda_{M}$ as $t \to \infty$, there exists $T = T(\delta) >0 $ such that for any $t \geq T$, 
 \begin{equation*}
	 \lambda_{M} \leq \mu[\psi(t)] \leq \lambda_{M} + \delta.
 \end{equation*}
 Therefore, it follows that for any $n>M$, and $t > T$
 \begin{equation*}
 \begin{aligned}
 |b_{n,k} (t)| &= |b_{n,k}(T)| \exp\left(\gamma \int_T^t \mu[\psi(\tau)] - \lambda_{n} \, d\tau \right) \\
 &< |b_{n,k}(T)| e^{-\gamma \delta (t - T)}.
 \end{aligned}
 \end{equation*}
 As a result, we obtain the convergence from
 \begin{equation*}
 \begin{aligned}
 	\left\| \psi(t,x) -\sum_{k=1}^{m_M}b_{M,k}(t) \phi_{M,k}(x)\right\|_{\Sigma} &= \left\| \sum_{n >M} \sum_{k=1}^{m_n} b_{n,k}(t) \phi_{n,k}(x) \right\|_{\Sigma} \\ 
 	& \leq \left\| \sum_{n >M} \sum_{k=1}^{m_n} b_{n,k}(T) \phi_{n,k}(x) \right\|_{\Sigma} e^{-\gamma \delta (t - T)}\\
 	&\lesssim e^{-\gamma \delta (t - T)} \| \psi(T,x)\|_{\Sigma}\rightarrow 0.
 \end{aligned}
\end{equation*}
as $t\rightarrow \infty$. In this way, we have shown that
\begin{equation*}
		\lim_{t \rightarrow \infty}\inf_{\varphi \in \mathcal W_{M}} \| \psi(t) - \varphi\|_{\Sigma} = 0. 
	\end{equation*}
\end{proof}
\begin{remark}
	 In general, the eigenvalues of $H_\Omega$ are not simple, and thus we cannot identify precisely the asymptotic state. In the case $\mu_\infty$ is simple, for example when $\mu_\infty$ is the smallest eigenvalue of $H_\Omega$, we obtain a strong convergence to a stationary state. 
\end{remark}

\section{Asymptotic Behavior in the Nonlinear Case}\label{sec:asNNGP}

In the previous section, we studied the asymptotic behavior when $g = 0$. In this case, different eigenspaces are invariant under the flow of \eqref{eq:gpe}. Thus the dissipation of the energy \eqref{eq:en_decrease} implies the convergence to the eigenspace of least energy in the decomposition of the initial datum. When $g \neq 0$, the power-type nonlinearity mixes the eigenspaces, and studying the asymptotic behavior becomes more complex.
\newline
In this section, we will first study the $\omega$-limit set for \eqref{eq:gpe} under the hypothesis of global well-posedness stated in Theorem \ref{thm:gwp}. For any $\psi_0 \in \Sigma(\R^d)$, the $\omega$-limit set is defined as follows.

\begin{defin}
 Let $\psi_0 \in \Sigma(\R^d)$ and let $\psi \in C[[0,\infty), \Sigma(\R^d))$ be the corresponding solution to \eqref{eq:gpe}. Then
 the $\omega$-limit set with respect to the metric induced by the $\Sigma$-norm is defined as
 \begin{equation*}
 \omega(\psi_0) = \left\{ u \in \Sigma: \, \exists \{t_n\} \subset \R^+, \, t_n \to \infty, \ \mbox{s.t.} \ \psi(t_n) \to u \ \mbox{in} \ \Sigma(\R^d) \right\}.
 \end{equation*}
\end{defin}

We shall also recall that a a stationary state $Q \in \Sigma(\R^d)$ is a solution to
\begin{equation}\label{eq:stst}
	0 = -\frac{1}{2}\Delta Q +VQ + g|Q|^{2\sigma} Q - \Omega \cdot LQ - \mu[Q] Q,
\end{equation}
and given a $\psi_0 \in \Sigma(\R^d)$, the set of all stationary states with the same mass as $\psi_0$ is denoted as
\begin{equation}
	\mathcal{S}(\psi_0) = \{u \in \Sigma: \, \| u\|_{L^2}=\| \psi_0\|_{L^2}, \, u \mbox{ solves } \eqref{eq:stst} \}. \label{eq:Ss}
\end{equation}
Proposition \ref{prp:gsGP} implies that there exists a function $Q_{gs} \in \mathcal S(\psi_0)$, referred to as the ground state, which minimizes the energy among all the elements of $\mathcal{S}$. We prove the following.

\begin{theorem}\label{thm:asyconv}
	Under the hypothesis of Theorem \ref{thm:gwp}, let $\psi_0 \in \Sigma(\R^d)$ and let  \  $\psi \in C([0,\infty),\Sigma(\R^d))$ be the corresponding solution to \eqref{eq:gpe}. Then there exists $Q \in \mathcal{S}(\psi_0)$ such that $Q \in \omega(\psi_0)$. 
	Moreover, we also have
	\begin{equation}\label{eq:enAsym}
	 	E[\psi(t)] \rightarrow E[Q] \ \mbox{ as } \ t \rightarrow \infty.
	\end{equation}
\end{theorem}

\begin{proof}
 We suppose that $\| \psi_0 \|_{L^2} =1$ without losing generality. Then \eqref{eq:mas} implies that $\| \psi(t)\|_{L^2} = 1$ for every $t \geq 0$ and 
 \begin{equation*}
	 \mu[\psi(t)] = E[\psi(t)] + \frac{g\sigma}{\sigma + 1} \| \psi(t)\|_{L^{2\sigma + 2}}^{2\sigma +2}. 
	\end{equation*}
 Notice that the energy $E[\psi(t)]$ is a continuous and decreasing function of time, see \eqref{eq:en_decrease}. It is also bounded from below by the energy of the ground state $E[Q_{gs}]$, see Proposition \ref{prp:gsGP}. As a consequence, there exists $E_\infty \geq E[Q_{gs}]$ such that 
	\begin{equation*}
	 \lim_{t \to \infty}	E[\psi(t)] = E_{\infty}.
	\end{equation*}
 By using \eqref{eq:en_decrease}, this implies that 
 \begin{equation*}
		E_\infty - E[\psi_0] = \lim_{t \rightarrow \infty}	(E[\psi(t)] - E[\psi_0]) = -2\gamma \lim_{t \rightarrow \infty}\int_0^t\| \partial_\tau \psi(\tau)\|_{L^2}^2 \, d\tau
	\end{equation*}
	and consequently $\partial_t \psi \in L^2([0,\infty),L^2(\R^d))$.
	Moreover, by a straightforward adaption of Proposition \ref{prp:gsGP} we obtain that $\mu[\psi(t)]$ is a continuous function bounded from below by 
	\begin{equation*}
	 \mu_{min} = \inf\left\{ \mu[u]: \, u\in \Sigma(\R^d), \ \| u\|_{L^2} = 1 \right\}.
	\end{equation*}
 Notice that for $g \geq 0$, \eqref{eq:en_decrease} implies $\mu[\psi(t)]$ is also straightforwardly bounded from above by $(1 + \sigma) E[\psi_0]$. For $g < 0$, by the Gagliardo-Nirenberg inequality, $\mu[\psi(t)]$ is still bounded from above by 
 \begin{equation*}
 \|\mu[\psi]\|_{L^\infty[0,\infty)} \lesssim E[\psi_0] + \|\psi\|_{L^\infty([0,\infty), \Sigma)}^{2\sigma + 2}
 \end{equation*}
 where we use that $\|\psi\|_{L^\infty([0,\infty), \Sigma)}^{2\sigma + 2} \lesssim 1$, see \eqref{eq:unifg}. 
Thus there exists a sequence $\{t_n\} \subset \R^+$, $t_n \rightarrow \infty$ as $n\rightarrow \infty$ such that	
\begin{equation*}
		\begin{cases}
			\psi(t_n) \rightharpoonup \psi_\infty \ \mbox{ in } \ \Sigma(\R^d), \\
			\partial_t \psi(t_n) \rightarrow 0 \ \mbox{in} \ L^2,\\
			\mu[\psi(t_n)] \rightarrow \mu_\infty.
		\end{cases}
	\end{equation*}
	Then it follows that
	\begin{equation*}
	 \begin{aligned}
	 0 &= (\ii - \gamma)\partial_t \psi + \frac{1}{2}\Delta \psi -V\psi - g|\psi|^{2\sigma} \psi + \Omega \cdot L\psi + \mu[\psi] \psi \\
	 &\rightharpoonup \frac{1}{2}\Delta \psi_\infty -V \psi_\infty - g|\psi_\infty|^{2\sigma} \psi_\infty + \Omega \cdot L \psi_\infty+ \mu_\infty \psi_\infty ,
	 \end{aligned}
	\end{equation*}
	where the convergence is intended in a weak sense. In particular, the profile $\psi_\infty$ satisfies weakly the stationary equation
	\begin{equation*}
		0 =-\frac{1}{2} \Delta \psi_\infty +V\psi_\infty + g|\psi_\infty|^{2\sigma} \psi_\infty - \Omega \cdot L \psi_\infty- \mu_\infty\psi_\infty.
	\end{equation*}
	Since $\psi_\infty \in \Sigma(\R^d)$, from the equation above, we obtain $\mu_\infty = \mu[\psi_\infty]$. The compact embedding $\Sigma(\R^d) \hookrightarrow L^2(\R^d) \cap L^{2\sigma + 2}(\R^d)$ (see Proposition \ref{prp:comemb}) implies that $\| \psi_\infty\|_{L^2} = \|\psi_0\|_{L^2} = 1$, $\psi_\infty \in \mathcal S(\psi_0)$ and 
	\begin{equation*}
	 \| \psi(t_n)\|_{L^{2\sigma + 2}}^{2\sigma +2} \to \| \psi_\infty\|_{L^{2\sigma + 2}}^{2\sigma +2}.
	\end{equation*}
	Thus, from the convergence 
 \begin{equation*}
 \mu[\psi(t_k)] = (H_\Omega \psi(t_k),\psi(t_k)) + g \| \psi(t_k)\|_{L^{2\sigma +2}}^{2\sigma + 2} \to \mu[\psi_\infty] = (H_\Omega \psi_\infty,\psi_\infty) + g \| \psi_\infty\|_{L^{2\sigma +2}}^{2\sigma + 2}
 \end{equation*}
 we obtain that 
	\begin{equation*}
	 (H_\Omega \psi(t_n),\psi(t_n)) \to (H_\Omega \psi_\infty,\psi_\infty).
	\end{equation*}
	Consequently, by using \eqref{eq:sigmaEn}, we obtain the strong convergence in $\Sigma(\R^d)$ 
	\begin{equation*}
	 \| \psi(t_n) - \psi_\infty\|_\Sigma^2 \lesssim (H_\Omega \left(\psi(t_n) - \psi_\infty\right),\psi(t_n) - \psi_\infty) \to 0.
	\end{equation*}
From the monotonicity of the energy \eqref{eq:en_decrease}, it also follows that
	\begin{equation*}
	 E_\infty = \lim_{t \to \infty} E[\psi(t)] = \lim_{n \to \infty } E[\psi(t_n)] = E[\psi_\infty].
	\end{equation*}
\end{proof}	

As a consequence of Theorem \ref{thm:asyconv}, we obtain the following properties of the $\omega$-limit associated with the evolutionary system \eqref{eq:gpe}.

\begin{prop}\label{prp:omega}
 Under the hypothesis of Theorem \ref{thm:gwp}, let $\psi_0 \in \Sigma(\R^d)$. Then $\omega(\psi_0)$ is connected and compact with respect to the $L^p$-norm for any $p \in [2, 2d/(d - 2)^+)$. Moreover, $\omega(\psi_0) \subset \mathcal{S}(\psi_0)$ and for any $u\in \omega (\psi_0)$, $E[u] = E_\infty$. 
\end{prop}

\begin{proof}
 Notice that for any $s>0$, the set 
 \begin{equation*}
 \{ \psi(t): \, t\geq s\}
 \end{equation*}
 is connected and relatively compact in $L^p(\R^d)$ for any $p\in \left[2, \frac{2d}{(d-2)^+}\right)$ (see Proposition \ref{prp:comemb}). As a consequence, 
 \begin{equation*}
 \omega(\psi_0) = \bigcap_{s>0} \overline{\{ \psi(t): \, t\geq s\}},
 \end{equation*}
 is connected and compact in $L^p(\R^d)$. Moreover, Theorem \ref{thm:asyconv} implies that there exists $Q\in \mathcal S(\psi_0)$ such that $Q \in \omega(\psi_0)$ and also that for any $u \in \omega(\psi_0)$, $E[u] = E[Q] = E_\infty$, see \eqref{eq:enAsym}. 
 \newline
 We now show that $\omega(\psi_0) \subset \mathcal{S}(\psi_0)$. Suppose by contradiction that there exists $v_0 \in \omega (\psi_0)$ such that $v_0 \notin \mathcal S(\psi_0)$. Let $\{t_n\} \subset \R^+$, $t_n \to \infty$ as $n\to \infty$ be such that $\psi(t_n) \to v_0$ in $\Sigma(\R^d)$. Let $v \in C([0,\infty),\Sigma(\R^d))$ be the solution to \eqref{eq:gpe} stemming from $v_0$.
 Then Theorem \ref{thm:asyconv} implies that there exists $\{\tau_n\} \subset \R^+$ be such that $v(\tau_n) \to W$ in $\Sigma(\R^d)$, where $W \in \mathcal S(\psi_0)$. Then it is straightforward to see that $\psi(t_n + \tau_n) \to W$ in $\Sigma(\R^d)$, thus $W \in \omega(\psi_0)$ and $E[W] = E[v_0] = E_\infty$. Thus \eqref{eq:en_decrease} implies that
 \begin{equation*}
 \int_0^\infty \| \partial_t v(\tau)\|_{L^2}^2 \, d\tau = E[W] - E[v_0] = 0,
 \end{equation*}
 that is $ \partial_t v = 0$ almost everywhere in $L^2$ and $v$ solves the stationary equation \eqref{eq:stst} in a weak sense for almost every $t >0$. Since $v \in \Sigma(\R^d)$, we obtain that $v \in \mathcal S(\psi_0)$, which is a contradiction.
\end{proof}

In general, neither Theorem \ref{thm:asyconv} nor Proposition \ref{prp:omega} imply the uniqueness of the asymptotic limit, nor that the solution tends to the limit for all $t>0$. The uniqueness could follow from the fact that $\omega(\psi_0)$ is connected. This would require showing that the set of stationary states with fixed $L^2$-norm and fixed energy is not connected (excluding the phase shift invariance), which is a result the authors are not aware of. 
\newline
On the other hand, there exist initial conditions that stem solutions where both uniqueness and convergence for all times can be proven. Indeed if the energy of the initial condition is small enough, for instance, smaller than that of the first excited state in $\mathcal{S}(\psi_0)$, then the solution must converge to the ground state at least on a sequence of times. We will now suppose that Conjecture \ref{concon} is true. We recall that this conjecture is partially confirmed by numerical approaches and formal expansion in \cite{BaRu18} and reference therein, although a general theoretical proof seems to be missing. 
\newline
Under the hypothesis that the conjecture is true, we can prove Theorem \ref{thm:conv}, that is, we can show that a solution with small enough initial energy converges strongly to a ground state.
\begin{proof}[Proof of Theorem \ref{thm:conv}:]
 Theorem \ref{thm:asyconv} and equation \eqref{eq:eneq} imply that there exists $\phi \in [0,2\pi)$ such that $e^{\ii\phi} Q_{gs} \in \omega(\psi_0)$, where $\|Q_{gs}\|_{L^2} = \| \psi_0\|_{L^2} $, and $Q_{gs}$ is given by Proposition \ref{prp:gsGP}. Without losing generality, we suppose $\phi = 0$. 
 \newline
 Suppose by contradiction that 
 \begin{equation*}
 \psi(t) \nrightarrow Q_{gs}, \quad \mbox{in} \quad \Sigma(\R^d)
 \end{equation*}
 as $t \to \infty$. Then there exists $\{t_n\} \subset \R^+$, $t_n \to \infty$ as $n\to \infty$ and $\varepsilon >0$ such that 
 \begin{equation*}
 \| \psi(t_n) - Q_{gs}\|_{\Sigma} \geq \varepsilon.
 \end{equation*}
 From the uniform boundness of $\| \psi(t_n)\|_\Sigma$, there exists a subsequence, still denoted by ${t_n}$ and $u \in \Sigma(\R^d)$ such that
 \begin{equation*}
 \psi(t_n) \rightharpoonup u \ \mbox{in} \ \Sigma.
 \end{equation*}
 From the compact embedding $\Sigma(\R^d) \hookrightarrow L^2(\R^d) \cap L^{2\sigma + 2}(\R^d)$, we obtain that
 \begin{equation*}
 \| u\|_{L^2} = \|\psi_0\|_{L^2}, \quad \| \psi(t_n)\|_{L^{2\sigma + 2}}^{2\sigma + 2} \to \| u \|_{L^{2\sigma + 2}}^{2\sigma + 2}.
 \end{equation*}
	As a consequence, from \eqref{eq:norms} and the weak lower semi-continuity of the norm we obtain
	\begin{equation*}
		E[u] \leq \liminf_n E[\psi(t_n)] = E[Q_{gs}].
	\end{equation*}
	Proposition \ref{prp:gsGP} implies that $E[u] = E[Q_{gs}]$. As a result, it follows from \eqref{eq:sigmaEn} that
	\begin{equation*}
	 \| \psi(t_n)\|_\Sigma \to \| Q_{gs}\|_\Sigma
	\end{equation*}
	that is $\psi(t_n) \rightarrow Q_{gs}$ in $\Sigma(\R^d)$. This is a contradiction. 
\end{proof}

\bibliographystyle{plain}
\bibliography{rotationalGP_bib}           

\end{document}